\ifdefined\cRKDGclean
  \documentclass[final,hidelinks]{siamart251216}
\else
  \documentclass[final,hidelinks]{siamart251216}
\fi
\ifdefined\cRKDGclean\else
\fi
\usepackage{amsfonts,amssymb,mathtools,microtype}
\usepackage{esint}
\usepackage{graphicx}
\usepackage{booktabs}
\usepackage{subcaption}
\usepackage{multirow}
\usepackage{flafter}
\usepackage{placeins}
\allowdisplaybreaks[1]

\newcommand{\dd}{\,\mathrm{d}}
\newcommand{\loc}{\mathcal{G}}
\newcommand{\DG}{\mathcal{F}}
\newcommand{\Q}{\mathcal{Q}}
\newcommand{\Bcal}{\mathcal{B}}
\newcommand{\wh}[1]{\widehat{#1}}
\newcommand{\bu}{\mathbf{u}}
\newcommand{\bv}{\mathbf{v}}
\newcommand{\bw}{\mathbf{w}}
\newcommand{\bfv}{\mathbf{f}}
\newcommand{\bn}{\mathbf{n}}
\newcommand{\bx}{\mathbf{x}}
\newcommand{\by}{\mathbf{y}}
\newcommand{\Reta}{R_\eta}

\newcommand{\norm}[1]{\lVert #1\rVert}

\theoremstyle{plain}
\theoremheaderfont{\normalfont\scshape}
\theorembodyfont{\normalfont\itshape}
\theoremseparator{.}
\theoremsymbol{}
\newtheorem{assumption}{Assumption}[section]
\theoremstyle{plain}
\theoremheaderfont{\normalfont\itshape}
\theorembodyfont{\normalfont}
\theoremseparator{.}
\theoremsymbol{}
\newtheorem{remark}{Remark}[section]
\theoremstyle{plain}
\theoremheaderfont{\normalfont\itshape}
\theorembodyfont{\normalfont}
\theoremseparator{.}
\theoremsymbol{}
\newtheorem{example}{Example}[section]
\crefname{assumption}{Assumption}{Assumptions}
\Crefname{assumption}{Assumption}{Assumptions}
\crefname{remark}{Remark}{Remarks}
\Crefname{remark}{Remark}{Remarks}
\hypersetup{pdftitle={Fully discrete multi-entropy-stable compact Runge--Kutta
discontinuous Galerkin methods for hyperbolic conservation laws},
 pdfsubject={Fully discrete multi-entropy-stable compact Runge--Kutta discontinuous Galerkin methods for hyperbolic conservation laws}}
\graphicspath{{figures/}}

\title{Fully discrete multi-entropy-stable compact Runge--Kutta
discontinuous Galerkin methods for hyperbolic conservation laws}
\author{Zheng Sun\thanks{Department of Mathematics, The University of Alabama,
		Tuscaloosa, AL 35487, USA. 
(\email{zsun30@ua.edu}).}}

\begin{document}
\raggedbottom
\maketitle
\begin{abstract}
We develop fully discrete multi-entropy-stable high-order compact
Runge--Kutta discontinuous Galerkin (cRKDG) methods for hyperbolic
conservation laws. Our main result is a weak cell entropy inequality
for multi-dimensional systems after one full high-order Runge--Kutta
(RK) time step. The proof exploits the cell-local intermediate stages of cRKDG:
their perturbations are controlled by the initial in-cell variation.
The analysis does not require a strong stability preserving
decomposition or stagewise entropy correction. For a sufficiently dissipative
Lax--Friedrichs flux and explicit RK methods with nonnegative output
weights, the inequality holds under a suitable time-step restriction
for strongly convex entropy pairs. For a prescribed, uniformly
conditioned family of such entropies, common dissipation and time-step
conditions ensure weak cell entropy stability for all entropy pairs.
A single cell-average-preserving scaling at the end of each step then
enforces all corresponding strong cell entropy inequalities
simultaneously. We also prove a Lax--Wendroff-type theorem: under the
stated assumptions, any strongly convergent sequence of numerical
solutions converges to a weak solution satisfying all prescribed
entropy inequalities. Numerical experiments in one and two dimensions
demonstrate high-order accuracy, the benefit of multiple entropy
constraints for nonconvex scalar problems, and robust performance on
shock problems.
\end{abstract}


\section{Introduction}
\label{sec:introduction}

Hyperbolic conservation laws arise in many applications, including
compressible gas dynamics, magnetohydrodynamics, and shallow water flows.
Even smooth initial data may lead to discontinuous solutions, and entropy
conditions play a fundamental role in selecting physically relevant weak
solutions. It is thus desirable that numerical methods reproduce suitable
entropy inequalities at the discrete level.

The discontinuous Galerkin (DG) finite element method is a widely used spatial discretization for hyperbolic conservation laws. It combines high-order accuracy,
local conservation, geometric flexibility, and a compact spatial
stencil \cite{CockburnShu2001}. When coupled with explicit Runge--Kutta
(RK) time integration, it gives the classical RKDG method. However,
the DG operator communicates with neighboring cells at every RK stage,
so the effective stencil of a full time step grows with the number of
stages.

To avoid this stage-by-stage stencil expansion, we  introduced the ``Runge--Kutta discontinuous
Galerkin method with compact stencils'' \cite{ChenSunXing2024}, which we refer to more briefly
as the compact RKDG (cRKDG) method. Its intermediate stages use a
cell-local spatial operator, while the conservative DG operator is used
only in the final update. An entire step therefore involves only
immediate neighbors, independently of the number of RK stages.
The method retains local conservation and the designed high-order
accuracy while simplifying data communication, parallel implementation and boundary treatment.
Techniques for bound preservation and total variation control were subsequently
developed in \cite{LiuSunZhang2025} and \cite{XuSun2026}, respectively.
In this work, we develop fully discrete multi-entropy-stable cRKDG methods.

At the semidiscrete level, the usual DG scheme can satisfy a cell
entropy inequality for the square entropy \cite{JiangShu1994}. For general convex entropies,
Tadmor's entropy
conservative and entropy stable flux framework
\cite{Tadmor1987,Tadmor2003} has led to many high-order entropy stable
methods. Examples include TeCNO schemes
\cite{FjordholmMishraTadmor2012}, summation-by-parts finite-difference
methods \cite{FisherCarpenter2013}, split-form DG methods
\cite{Gassner2013}, and DG methods based on suitable quadrature and flux
differencing \cite{ChenShu2017,Chan2018}; see \cite{ChenShu2020} for a
review. Other approaches use local entropy corrections or artificial
viscosity \cite{AbgrallOffnerRanocha2022,LiuGuoJiangZhang2025,Chan2025AV},
or impose semidiscrete entropy constraints through DG subcell limiting
\cite{RuedaRamirezBolmKuzminGassner2024,LinChan2024}.

For high-order explicit schemes, establishing fully discrete entropy
stability of the unmodified method can be challenging even for the
square entropy. See \cite{SunShu2019,XuZhangShuWang2019} for linear
problems and \cite{WeiShu2025,WeiShu2026} for nonlinear problems. For general convex entropies, limiting, correction, and relaxation
provide ways to enforce fully discrete entropy stability. A common idea
is to modify a high-order candidate relative to a suitable reference
state. In convex limiting, a property-preserving low-order discretization
supplies this reference \cite{GuermondPopov2016,GuermondEtAlGreedy2024}.
The high-order discretization is blended with this reference through
conservative flux limiting \cite{GuermondPopovTomas2019,Kuzmin2020,Hajduk2021}.
Related approaches are used in DG subcell methods
\cite{Pazner2021,RuedaRamirezPaznerGassner2022}.
Suitable entropy constraints on the flux corrections yield fully discrete
entropy stable schemes \cite{KuzminHajdukRupp2022}.
Relaxation RK methods instead use the solution from the previous time step as the reference
state and rescale the RK increment to preserve prescribed global entropy
structure \cite{RanochaEtAl2020}. Their local extensions enforce cellwise
inequalities for several convex entropies using suitable spatial entropy
estimates \cite{RanochaDalcinParsani2020}.

For DG methods, a natural question is whether the updated cell average
can itself serve as the reference state \cite{ZhangShu2010}. If it satisfies a
\emph{weak cell entropy inequality}, convexity allows the polynomial to
be scaled towards it to recover the full cell entropy inequality.
This gives a simple cellwise correction based on the high-order update
itself and preserves the conservative averages.

Carlier and Renac \cite{CarlierRenac2023} established this weak property
for forward-Euler updates of high-order spectral discontinuous
approximations. Building on this observation, we developed a
limiter-based approach \cite{LiuGuoJiangSun} that scales the polynomials to
enforce fully discrete cell entropy inequalities for several entropies
simultaneously. In concurrent work, Wu \cite{WuEPO} describes this
mechanism from a \emph{weak-to-strong} viewpoint and combines entropy,
invariant-domain, and oscillation constraints in a common framework.
Related fully discrete methods have been developed for  MHD equations \cite{LiuJiangSun2026,CaoWuYuan2026} and relativistic
Euler equations \cite{XuWu2026}.

Within this framework, one approach for extension to RK time integration is to
apply the forward-Euler entropy argument stage by stage through a
strong stability preserving (SSP) decomposition. The resulting
stagewise corrections may reduce the scheme to first order once
the limiter is active \cite{LiuGuoJiangSun,WuEPO}.
SSP multistep methods apply the correction only at the end of each step
\cite{LiuGuoJiangSun,WuEPO}, but require solution history and startup
procedures. Moreover, variable time stepping can be less convenient
because the SSP condition depends on recent step sizes
\cite{HadjimichaelKetchesonLocziNemeth2016}. One-step RK methods instead
naturally accommodate the changing Courant--Friedrichs--Lewy (CFL) time
step as wave speeds evolve.

We address this difficulty for cRKDG by establishing the weak cell
entropy inequality after one full high-order RK step, without an
SSP decomposition or stagewise entropy correction. At first glance,
the cell-local intermediate stages seem unfavorable for entropy
analysis because they lack the usual conservative intercell flux
structure. However, their locality indeed simplifies the analysis.
Since neighboring information is not introduced at the intermediate
stages, the stage perturbations can be controlled by the variation
already present in the same cell at the beginning of the step.
Neighboring states enter only through the final conservative coupling.
Thus, the compact stage structure is not only a computational
advantage; it also provides a direct route to applying the
weak-to-strong framework to high-order cRKDG schemes without
an SSP requirement.

Our main result holds for a sufficiently dissipative Lax--Friedrichs
flux and explicit RK methods with nonnegative output weights under a
suitable CFL condition. It applies to multi-dimensional systems with
strongly convex entropy pairs under the stated regularity, curvature,
and state-range assumptions. The proof combines the Jensen
gap of the input polynomial, relative entropy estimates for the
numerical flux, and control of the local stages. In summary, the main
contributions are:
\begin{itemize}
    \item a weak cell entropy inequality after one full high-order cRKDG
    time step for multi-dimensional systems;

    \item a proof mechanism exploiting the cell-local intermediate
    stages, without an SSP decomposition or stagewise entropy correction;

    \item fully discrete multi-entropy stability for a prescribed,
    uniformly conditioned family of strongly convex entropy pairs,
    obtained through one cell-average-preserving end-of-step scaling;

    \item a Lax--Wendroff-type theorem showing that, under the stated
    assumptions, strongly convergent numerical solutions have a limit
    that is a weak solution satisfying the prescribed entropy inequalities.
\end{itemize}

The rest of the paper is organized as follows.
Section~\ref{sec:setup} introduces cRKDG and states the weak cell entropy
inequality, proved in Section~\ref{sec:proof}.
Section~\ref{sec:entropy-limiter} presents the multi-entropy extension,
the entropy correction, and the Lax--Wendroff theorem, whose proof is
given in Appendix~\ref{app:lax-wendroff}.
Section~\ref{sec:numerical-experiments} reports numerical experiments,
and Section~\ref{sec:conclusion} concludes the paper.

\section{cRKDG formulation and weak entropy stability}
\label{sec:setup}

This section introduces the cRKDG formulation and the assumptions used throughout the entropy analysis. The weak cell entropy inequality is stated in Subsection~\ref{sec:weak-entropy}.

\subsection{Conservation laws and entropy pairs}
\label{subsec:conservation-law}

Consider the hyperbolic system
\begin{equation}\label{eq:system}
    \frac{\partial\bu}{\partial t}
    +\sum_{r=1}^d\frac{\partial\bfv_r(\bu)}{\partial x_r}
    =\mathbf0,
    \qquad \bx\in\Omega\subset\mathbb R^d,\quad 0<t\leq T.
\end{equation}
Here $\bu\in\mathbb R^m$ is the vector of conserved variables,
$\bfv_r$ is the physical flux in the $x_r$ direction, and $\Omega$
is a bounded domain. We assume periodic boundary conditions or that the solution remains
supported in a fixed compact subset of $\Omega$ for all $t\in[0,T]$.

A convex scalar function $\eta$ and a vector-valued function
$\boldsymbol{\mathcal Q}=(\Q_1,\ldots,\Q_d)$ form an entropy pair for
\eqref{eq:system} if
\begin{equation}\label{eq:entropy-pair}
    \Q_r'(\bu)=\eta'(\bu)\bfv_r'(\bu),
    \qquad r=1,\ldots,d.
\end{equation}
Derivatives of scalar functions with respect to $\bu$ are row vectors,
and $\bfv_r'(\bu)$ denotes the $m\times m$ flux Jacobian. The associated entropy inequality, understood in the sense of distributions, is
\begin{equation}\label{eq:entropy-inequality}
    \frac{\partial\eta(\bu)}{\partial t}
    +\sum_{r=1}^d\frac{\partial\Q_r(\bu)}{\partial x_r}\leq0.
\end{equation}

Unless otherwise stated, all integrals are exact in the following sections. 

\subsection{Mesh, finite element spaces, and local norms}
\label{subsec:dg-space}

Let $\mathcal T_h$ be a conforming, uniformly shape-regular mesh of
$\Omega$ consisting of simplices or Cartesian boxes. For
$K\in\mathcal T_h$, define
$h_K:=\operatorname{diam}(K)$, $|K|:=\int_K1\dd\bx$, and $|\partial K|:=\int_{\partial K}1\dd s$.
In one dimension, boundary integrals are sums over the two endpoints,
and $|\partial K|=2$.

Let $\mathbb Z^k(K)$ be the total-degree polynomial
space $\mathbb P^k(K)$ or the tensor-product polynomial space
$\mathbb Q^k(K)$, with degree at most $k$ in each coordinate.
The DG space is
\begin{equation}\label{eq:space}
    \mathbb V_h^k
    :=\left\{\bv_h\in[L^2(\Omega)]^m:
       \bv_h|_K\in[\mathbb Z^k(K)]^m
       \text{ for all }K\in\mathcal T_h\right\}.
\end{equation}
For a face $e\subset\partial K$ with outward unit normal
$\bn_K=(n_{K,1},\ldots,n_{K,d})$, the exterior ($+$) and interior ($-$)
traces are defined by
$\bv_h^\pm(\bx):=\lim_{\varepsilon\to0^+}
\bv_h(\bx\pm\varepsilon\bn_K)$, $\bx\in e$. On physical boundary faces, exterior numerical states are supplied
by the boundary conditions. 
For integrable functions, define the integral average and the local
$L^2$ norms by
\begin{equation}\label{eq:cell-norm}
\begin{aligned}
    \fint_K g\dd\bx&:=\frac1{|K|}\int_K g\dd\bx,
    &\overline{\bv}_K&:=\fint_K\bv\dd\bx,\\
    \norm{\bv}_K^2&:=\int_K|\bv|^2\dd\bx,
    &\norm{\mathbf g}_{\partial K}^2
        &:=\int_{\partial K}|\mathbf g|^2\dd s.
\end{aligned}
\end{equation}
The pointwise norm $|\cdot|$ is the Euclidean norm for vectors and the
Frobenius norm for matrices. In particular,
$|\nabla\bv|^2=\sum_{r=1}^d|\partial_{x_r}\bv|^2$.
For a cell polynomial, the boundary norm uses its interior trace.
For a DG function, the two traces are distinguished by
$\norm{\bv_h^-}_{\partial K}$ and $\norm{\bv_h^+}_{\partial K}$.

Recall the inverse and trace estimates for finite elements. There is a common constant $C_k>0$ such that, for every
$K\in\mathcal T_h$ and $\mathbf z\in[\mathbb Z^k(K)]^m$,
\begin{equation}\label{eq:inverse-trace}
    h_K\norm{\nabla\mathbf z}_K\leq C_k\norm{\mathbf z}_K,
    \qquad
    h_K\norm{\mathbf z}_{\partial K}^2
        \leq C_k\norm{\mathbf z}_K^2.
\end{equation}
The constant depends only on $k$, $d$, the reference element, and the
shape-regularity bound, and is independent of the mesh size and time step.
Taking $\mathbf z$ to be a unit constant vector gives
\begin{equation}\label{eq:cell-surface-ratio}
    \frac{h_K|\partial K|}{|K|}\leq C_k.
\end{equation}

\subsection{The cRKDG method}
\label{subsec:crkdg-method}

Let $a_{\ell j}$ and $b_\ell$ be the coefficients of an $s$-stage
explicit RK method of classical order $k+1$, with
$a_{\ell j}=0$ for $j\geq\ell$. Given $\bu_h^n\in\mathbb V_h^k$
and a time step $\tau>0$, the cRKDG method uses a cell-local
operator $\loc$ for its intermediate stages and a conservative
DG operator $\DG$ for the pre-limiter solution:
\begin{subequations}\label{eq:general-crkdg}
\begin{align}
    \bu_h^{n,(1)}&=\bu_h^n,
        \label{eq:general-crkdg-a}\\
    \bu_h^{n,(\ell)}&=\bu_h^n
       +\tau\sum_{j=1}^{\ell-1}a_{\ell j}\loc(\bu_h^{n,(j)}),
       \qquad \ell=2,\ldots,s,
        \label{eq:general-crkdg-b}\\
    \bu_h^{n,(s+1)}:=\bu_h^{n+1,\star}&=\bu_h^n
       +\tau\sum_{\ell=1}^s b_\ell\DG(\bu_h^{n,(\ell)}).
        \label{eq:general-crkdg-c}
\end{align}
\end{subequations}
The operators $\loc$ and $\DG$ take values in $\mathbb V_h^k$
and are defined cellwise by
\begin{subequations}\label{eq:operators}
\begin{align}
    \int_K\loc(\bu_h)\cdot\bv_h\dd\bx
    &=\sum_{r=1}^d\int_K
        \bfv_r(\bu_h)\cdot\partial_{x_r}\bv_h\dd\bx
       -\int_{\partial K}
        \bfv_{\bn_K}(\bu_h^-)\cdot\bv_h^-\dd s,
        \label{eq:local-operator}\\
    \int_K\DG(\bu_h)\cdot\bv_h\dd\bx
    &=\sum_{r=1}^d\int_K
        \bfv_r(\bu_h)\cdot\partial_{x_r}\bv_h\dd\bx
       -\int_{\partial K}
        \wh{\bfv}_{\bn_K}(\bu_h^-,\bu_h^+)\cdot\bv_h^-\dd s
        \label{eq:dg-operator}
\end{align}
\end{subequations}
for every $\bv_h\in[\mathbb Z^k(K)]^m$ and $K\in\mathcal T_h$.
Here, for a unit vector $\bn=(n_1,\ldots,n_d)$, 
$\bfv_{\bn}:=\sum_{r=1}^d n_r\bfv_r$ and 
$\wh{\bfv}_{\bn}$ is the Lax--Friedrichs numerical flux
\begin{equation}\label{eq:lf-flux}
    \wh{\bfv}_{\bn}(\bw,\bv)
       :=\frac12[\bfv_{\bn}(\bw)+\bfv_{\bn}(\bv)]
           -\frac\alpha2(\bv-\bw),
\end{equation}
where $\alpha>0$ is fixed throughout the current time step and satisfies
the condition in \eqref{eq:alpha-system}. It may be updated between time
steps, provided the same value is used in all stages and in both numerical
fluxes during each step.
The two operators in \eqref{eq:operators} differ only in their boundary terms:
$\loc$ uses the interior physical flux, whereas $\DG$ uses
the numerical flux and involves neighboring traces.

Testing \eqref{eq:dg-operator} with constant vectors and averaging
\eqref{eq:general-crkdg-c} yields the conservative cell-average update
\begin{equation}\label{eq:crkdg-mean-update}
    \overline{\bu}_{h,K}^{n+1,\star}
       =\overline{\bu}_{h,K}^{\,n}
        -\frac\tau{|K|}\int_{\partial K}
            \wh{\bfv}_{\bn_K}^{n,\mathrm{RK}}\dd s,
\end{equation}
where $\wh{\bfv}_{\bn_K}^{n,\mathrm{RK}}$ is the RK-averaged
numerical flux
\begin{equation}\label{eq:rk-physical-flux-general}
    \wh{\bfv}_{\bn_K}^{n,\mathrm{RK}}
       :=\sum_{\ell=1}^s b_\ell\wh{\bfv}_{\bn_K}
          (\bu_h^{n,(\ell),-},\bu_h^{n,(\ell),+}).
\end{equation}

\subsection{Proper numerical entropy fluxes}
\label{subsec:numerical-entropy-fluxes}

For $\bn = (n_1,\cdots,n_d)$, define the normal entropy flux by
$\Q_{\bn}:=\sum_{r=1}^d n_r\Q_r$. The compatibility condition
\eqref{eq:entropy-pair} gives
\begin{equation}\label{eq:normal-entropy-pair}
    \Q_{\bn}'(\bv)=\eta'(\bv)\bfv_{\bn}'(\bv).
\end{equation}
A numerical entropy flux $\wh{\Q}_{\bn}$ associated with a numerical
physical flux $\wh{\bfv}_{\bn}$ is called proper
\cite{KarlsenRisebroStorrost2014,ZhaoWu1997} if
$\partial_1\wh{\Q}_{\bn}(\bw,\bv)
       =\eta'(\bw)\,\partial_1\wh{\bfv}_{\bn}(\bw,\bv)$ and   $\partial_2\wh{\Q}_{\bn}(\bw,\bv)
       =\eta'(\bv)\,\partial_2\wh{\bfv}_{\bn}(\bw,\bv)$. For the Lax--Friedrichs flux \eqref{eq:lf-flux}, the corresponding
proper numerical entropy flux is
\begin{equation}\label{eq:lf-entropy-flux}
    \wh{\Q}_{\bn}(\bw,\bv)
       :=\frac12[\Q_{\bn}(\bw)+\Q_{\bn}(\bv)]
           -\frac\alpha2[\eta(\bv)-\eta(\bw)],
\end{equation}
with the same parameter $\alpha$. It is also consistent and conservative:
\begin{equation}\label{eq:normal-entropy-flux-conservation}
    \wh{\Q}_{\bn}(\bv,\bv)=\Q_{\bn}(\bv),
    \qquad
    \wh{\Q}_{-\bn}(\bv,\bw)=-\wh{\Q}_{\bn}(\bw,\bv).
\end{equation}
Using the stages in \eqref{eq:general-crkdg}, define the RK-averaged
numerical entropy flux on faces by
\begin{equation}\label{eq:rk-entropy-flux-general}
    \wh{\Q}_{\bn_K}^{n,\mathrm{RK}}
       :=\sum_{\ell=1}^s b_\ell\wh{\Q}_{\bn_K}
          (\bu_h^{n,(\ell),-},\bu_h^{n,(\ell),+}).
\end{equation}

\subsection{Assumptions}
\label{subsec:assumptions}

The entropy analysis uses the following assumptions.

\begin{assumption}[Regularity and state range]
\label{ass:state-range}
There exists a nonempty compact convex set $\Bcal\subset\mathbb R^m$
such that $\bfv_r$, $\eta$, and $\Q_r$ are $C^2$ on an open
neighborhood of $\Bcal$, for $r=1,\ldots,d$. At every step and in
every cell $K\in\mathcal T_h$, the input, local stages, and pre-limiter
solution satisfy
\begin{equation}\label{eq:numerical-state-bounds}
    \bu_{h,K}^{n,(\ell)}(\bx)\in\Bcal
       \quad\text{for }\bx\in\overline K,\ \ell=1,\ldots,s+1.
\end{equation}
The set $\Bcal$ is independent of the mesh and the time step. The exterior states supplied at physical boundary faces are also
assumed to belong to $\mathcal{B}$ at every RK stage.
\end{assumption}

\begin{assumption}[Uniform convexity of the entropy]
\label{ass:entropy}
There exist constants $0<\mu_\eta\leq\nu_\eta<\infty$, independent
of the mesh and time step, such that
\begin{equation}\label{eq:hessian-bounds}
    \mu_\eta I_m\preceq\eta''(\bv)\preceq\nu_\eta I_m,
    \qquad \bv\in\Bcal.
\end{equation}
Here $\preceq$ denotes the ordering of symmetric matrices by positive
semidefiniteness.
\end{assumption}

\begin{assumption}[Nonnegative RK output weights]
\label{ass:rk-weights}
The output weights satisfy
\begin{equation}\label{eq:rk-output-weights}
    b_\ell\geq0,\qquad \sum_{\ell=1}^s b_\ell=1.
\end{equation}
\end{assumption}

\begin{assumption}[Sufficient flux dissipation]
\label{ass:alpha}
The parameter $\alpha$ in \eqref{eq:lf-flux} satisfies
\begin{equation}\label{eq:alpha-system}
    \alpha\geq2\kappa_\eta L_f,
    \qquad \text{where}\quad
    \kappa_\eta:=\frac{\nu_\eta}{\mu_\eta},
    \quad
    L_f:=\max_{\bv\in\Bcal}
    \left(\sum_{r=1}^d\norm{\bfv_r'(\bv)}_2^2\right)^{1/2}.
\end{equation}
\end{assumption}

For $\bv\in\Bcal$, arbitrary vectors $\mathbf z_r\in\mathbb R^m$,
and $|\bn|=1$, Cauchy--Schwarz gives
\begin{equation}\label{eq:directional-flux-bounds}
    \norm{\bfv_{\bn}'(\bv)}_2\leq L_f,
    \qquad
    \left|\sum_{r=1}^d\bfv_r'(\bv)\mathbf z_r\right|
    \leq L_f\left(\sum_{r=1}^d|\mathbf z_r|^2\right)^{1/2}.
\end{equation}

\begin{remark}
The set $\Bcal$ bounds the numerical states and the derivatives used below;
it need not coincide with the invariant region. For example, for Burgers'
equation with sine initial data, one may take $\Bcal=[-3,3]$ instead of
$[-1,1]$, provided the numerical states lie in $\Bcal$.
\end{remark}
\subsection{Main theorem: weak cell entropy inequality}
\label{sec:weak-entropy}

Define
\begin{equation}\label{eq:Ck-astar}
    a_*:=\max_{1\leq\ell\leq s}
                     \sum_{j=1}^{\ell-1}|a_{\ell j}|,
    \qquad
    \lambda_K:=\frac{\tau}{h_K},
    \qquad \sigma_K:=\alpha\lambda_K.
\end{equation}

\begin{theorem}[Weak cell entropy inequality]
\label{thm:weak-entropy}
Under Assumptions~\ref{ass:state-range}, \ref{ass:entropy}, 
\ref{ass:rk-weights}, and \ref{ass:alpha},
suppose
\begin{equation}\label{eq:higher-order-CFL}
    \sigma_K=\frac{\alpha\tau}{h_K}
    \leq{(C_k\max\{a_*,9\kappa_\eta\})^{-1}},
    \qquad K\in\mathcal T_h.
\end{equation}
Then the pre-limiter cRKDG update satisfies
\begin{equation}\label{eq:higher-order-weak}
    \eta\!\left(\overline{\bu}_{h,K}^{n+1,\star}\right)
    \leq\fint_K\eta(\bu_h^n)\dd\bx
       -\frac{\tau}{|K|}\int_{\partial K}
                     \wh{\Q}_{\bn_K}^{n,\mathrm{RK}}\dd s.
\end{equation}
\end{theorem}

The inequality is considered weak because it controls the entropy of the new cell
average, rather than the cell average of the entropy of the pre-limiter solution.

\begin{remark}[Lax--Friedrichs constants and CFL conditions]
Assumption \ref{ass:alpha} and the time step constraint \eqref{eq:higher-order-CFL} provide only sufficient conditions for Theorem \ref{thm:weak-entropy}. In practice, the weak cell entropy inequality may hold for a smaller $\alpha$ or a larger $\tau$. We therefore do not enforce these a priori bounds in the numerical
experiments, but instead verify the entropy inequality a posteriori
using a trial-step accept/reject procedure. See Section \ref{sec:numerical-experiments} for details.
\end{remark}

\begin{remark}[Quadrature-defined operators]
\label{rem:quadrature-entropy}Theorem~\ref{thm:weak-entropy} extends to formulations in which exact
integration is replaced by appropriate quadrature rules in the
evaluation of $\loc$ and $\DG$.
For $\mathbb P^k$ elements, let $\mathsf Q_K$ and $\mathsf Q_e$
denote the volume and face quadrature functionals, respectively.
Suppose that both rules have positive weights and are exact through
degree $2k$. Under \eqref{eq:higher-order-CFL}, with $C_k$ enlarged
if necessary, we have
\begin{equation}\label{eq:quadrature-weak}
    \eta\!\left(\overline{\bu}_{h,K}^{n+1,\star}\right)
    \leq\frac1{|K|}\mathsf Q_K[\eta(\bu_h^n)]
       -\frac\tau{|K|}\sum_{e\subset\partial K}
          \mathsf Q_e[\wh{\Q}_{\bn_K}^{n,\mathrm{RK}}].
\end{equation}
The proof follows the argument in Section~\ref{sec:proof},
with a modified proof of Lemma~\ref{lem:local-operator-general-k}.
The multi-entropy results extend similarly, using quadrature-based
entropy budgets and scaled polynomials.
\end{remark}

\section{Proof of the weak cell entropy inequality}\label{sec:proof}

We now prove the main theorem in Subsection~\ref{sec:weak-entropy}.
The argument begins with a three-term entropy decomposition and then
estimates each contribution separately. 

Throughout this section, we fix a time step and a cell $K$, and suppress the superscript $n$ and subscript $h$: $\bu$ and
$\bu^{(\ell)}$ denote the input polynomial and its local stages, and  $\overline{\bu}_K$ and $\overline{\bu}_K^{n+1,\star}$ are the input and
new cell averages. 

\subsection{The three-term entropy decomposition}

For stage $\ell$, define the scaled outward numerical-flux sums and their RK averages
\begin{subequations}\label{eq:flux-sums}
\begin{alignat}{2}
    \wh{\mathbf F}_K^{(\ell)}
    &:=\frac{h_K}{|K|}\int_{\partial K}
       \wh{\bfv}_{\bn_K}(\bu^{(\ell),-},\bu^{(\ell),+})\dd s,\qquad &\wh{\mathbf F}_K^{\mathrm{RK}}
       :=\sum_{\ell=1}^s b_\ell\wh{\mathbf F}_K^{(\ell)},\label{eq:stage-flux-sums}\\
    \wh{Q}_K^{(\ell)} 
    &:=\frac{h_K}{|K|}\int_{\partial K}
       \wh{\Q}_{\bn_K}(\bu^{(\ell),-},\bu^{(\ell),+})\dd s, \qquad & \wh{Q}_K^{\mathrm{RK}}
       :=\sum_{\ell=1}^s b_\ell\wh{Q}_K^{(\ell)}.\label{eq:entropy-flux-sums}
\end{alignat}
\end{subequations}
Using \eqref{eq:crkdg-mean-update}, we obtain
\begin{equation}\label{eq:higher-order-mean-update}
    \overline{\bu}_K^{n+1,\star}
    =\overline{\bu}_K-\lambda_K\wh{\mathbf F}_K^{\mathrm{RK}}.
\end{equation}
Both means lie in $\Bcal$. Taylor's theorem and \eqref{eq:hessian-bounds}
give, for some $\bw_K$ between them,

\begin{equation}\label{eq:higher-order-final-taylor}
\begin{aligned}
    \eta(\overline{\bu}_K^{n+1,\star})
    &=\eta(\overline{\bu}_K)
      -\lambda_K\eta'(\overline{\bu}_K)\wh{\mathbf F}_K^{\mathrm{RK}} +\frac{\lambda_K^2}{2}
      \bigl(\wh{\mathbf F}_K^{\mathrm{RK}}\bigr)^{\mathsf T}
      \eta''(\bw_K)
      \wh{\mathbf F}_K^{\mathrm{RK}}\\
    &\leq\eta(\overline{\bu}_K)
      -\lambda_K\eta'(\overline{\bu}_K)\wh{\mathbf F}_K^{\mathrm{RK}}
      +\frac{\nu_\eta}{2}\lambda_K^2
        |\wh{\mathbf F}_K^{\mathrm{RK}}|^2.
\end{aligned}
\end{equation}
Subtract the old cell entropy and add the RK entropy-flux sum:
\begin{equation}\label{eq:master-decomposition-expanded}
\begin{aligned}
 &\eta(\overline{\bu}_K^{n+1,\star})
       -\fint_K\eta(\bu)\dd\bx
       +\lambda_K\wh{Q}_K^{\mathrm{RK}}\\
 &\leq
 -\underbrace{\left[\fint_K\eta(\bu)\dd\bx
                    -\eta(\overline{\bu}_K)\right]}
              _{\mathclap{\text{I: Jensen gap}}}
 +\lambda_K\!\underbrace{
   \vphantom{\left[\fint_K\eta(\bu)\dd\bx
                    -\eta(\overline{\bu}_K)\right]}
   \left[\wh{Q}_K^{\mathrm{RK}}
       -\eta'(\overline{\bu}_K)\wh{\mathbf F}_K^{\mathrm{RK}}\right]}
              _{\mathclap{\substack{\text{II: relative entropy-flux}\\
                                    \text{contribution}}}}
 +\frac{\nu_\eta}{2}\lambda_K^2\!\underbrace{
   \vphantom{\left[\fint_K\eta(\bu)\dd\bx
                    -\eta(\overline{\bu}_K)\right]}
   |\wh{\mathbf F}_K^{\mathrm{RK}}|^2}
              _{\mathclap{\substack{\text{III: squared}\\
                                    \text{numerical-flux sum}}}}.
\end{aligned}
\end{equation}
The Jensen gap is nonnegative by convexity. We estimate the two
interface terms separately and then control their interior-trace
contributions by the Jensen gap.

\subsection{Term I: the Jensen gap}

Since $\bu-\overline{\bu}_K$ has zero mean and $\eta(\overline{\bu}_K)$ and $\eta'(\overline{\bu}_K)$ are constant, the Jensen gap can be written as
\begin{equation}\label{eq:Jgap-relative-entropy}
    J_K(\bu):=\fint_K\eta(\bu)\dd\bx-\eta(\overline{\bu}_K)
       =\fint_K\Reta(\bu(\bx),\overline{\bu}_K)\dd\bx,
\end{equation}
where $\Reta$ is the relative entropy \cite{Dafermos2016}, i.e., the Bregman divergence generated by $\eta$:
\begin{equation}\label{eq:Reta-def}
    \Reta(\bw,\bv):=\eta(\bw)-\eta(\bv)-\eta'(\bv)(\bw-\bv),
    \qquad \bw,\bv\in\Bcal.
\end{equation}

\begin{lemma}[Quadratic bounds for relative entropy]
\label{lem:relative-remainders}
For every $\bw,\bv\in\Bcal$,
\begin{equation}\label{eq:Reta-bounds}
    \frac{\mu_\eta}{2}|\bw-\bv|^2
    \leq\Reta(\bw,\bv)
    \leq\frac{\nu_\eta}{2}|\bw-\bv|^2.
\end{equation}
\end{lemma}
\begin{proof}
Set $\boldsymbol\delta=\bw-\bv$. By the fundamental theorem of calculus,
\begin{equation}\label{eq:Reta-integral}
    \Reta(\bw,\bv)=\int_0^1
       [\eta'(\bv+t\boldsymbol\delta)-\eta'(\bv)]\boldsymbol\delta\dd t.
\end{equation}
Convexity of $\Bcal$ and \eqref{eq:hessian-bounds} give
$\mu_\eta t|\boldsymbol\delta|^2\leq
[\eta'(\bv+t\boldsymbol\delta)-\eta'(\bv)]\boldsymbol\delta
\leq\nu_\eta t|\boldsymbol\delta|^2$. Integrating in $t$ proves the bounds.
\end{proof}

\begin{lemma}[The Jensen gap controls the in-cell variation]
\label{lem:jensen-gap}
\begin{equation}\label{eq:J-bounds}
    \frac{\mu_\eta}{2|K|}\norm{\bu-\overline{\bu}_K}_K^2
    \leq J_K(\bu)
    \leq\frac{\nu_\eta}{2|K|}\norm{\bu-\overline{\bu}_K}_K^2.
\end{equation}
\end{lemma}
\begin{proof}
Average \eqref{eq:Reta-bounds} over $K$ with $\bw=\bu(\bx)$ and
$\bv=\overline{\bu}_K$, and use \eqref{eq:Jgap-relative-entropy}.
\end{proof}

\subsection{Term II: the relative entropy-flux contribution}

\subsubsection{Reference-flux cancellation and two-state bounds}

The divergence theorem applied to constant coordinate vector
fields gives $\int_{\partial K}\bn_K\dd s=\mathbf0$.
For every constant state $\bv\in\Bcal$, it follows that $\int_{\partial K}\bfv_{\bn_K}(\bv)\dd s=\mathbf0$ and $\int_{\partial K}\Q_{\bn_K}(\bv)\dd s=0$.
Inserting these identities with $\bv=\overline{\bu}_K$ into
the stage entropy-flux contribution yields
\begin{equation}\label{eq:quadratic-interface-reference}
\begin{aligned}
 \wh{Q}_K^{(\ell)}
       -\eta'(\overline{\bu}_K)\wh{\mathbf F}_K^{(\ell)}
 =&\frac{h_K}{|K|}\int_{\partial K}
   \Bigl\{\wh{\Q}_{\bn_K}(\bu^{(\ell),-},\bu^{(\ell),+})
                         -\Q_{\bn_K}(\overline{\bu}_K)\\
 &\quad-\eta'(\overline{\bu}_K)
      [\wh{\bfv}_{\bn_K}(\bu^{(\ell),-},\bu^{(\ell),+})
                   -\bfv_{\bn_K}(\overline{\bu}_K)]\Bigr\}\dd s.
\end{aligned}
\end{equation}
Substituting the Lax--Friedrichs formulas into the integrand gives
\begin{equation}\label{eq:LF-relative-quadratic-decomposition}
\begin{aligned}
 &\wh{\Q}_{\bn_K}(\bu^{(\ell),-},\bu^{(\ell),+})
       -\Q_{\bn_K}(\overline{\bu}_K)
       -\eta'(\overline{\bu}_K)
       [\wh{\bfv}_{\bn_K}(\bu^{(\ell),-},\bu^{(\ell),+})
        -\bfv_{\bn_K}(\overline{\bu}_K)]\\
 &\quad=\frac12[\alpha\Reta(\bu^{(\ell),-},\overline{\bu}_K)
                  +R_{\Q_{\bn_K}}(\bu^{(\ell),-},\overline{\bu}_K)]\\
 &\quad\ -\frac12[\alpha\Reta(\bu^{(\ell),+},\overline{\bu}_K)
                  -R_{\Q_{\bn_K}}(\bu^{(\ell),+},\overline{\bu}_K)],
\end{aligned}
\end{equation}
where the relative entropy $\Reta$ is defined in \eqref{eq:Reta-def} and the relative entropy flux in direction $\bn$ is
\begin{equation}\label{eq:RQ-def}
    R_{\Q_{\bn}}(\bw,\bv)
       :=\Q_{\bn}(\bw)-\Q_{\bn}(\bv)
         -\eta'(\bv)[\bfv_{\bn}(\bw)-\bfv_{\bn}(\bv)],
    \qquad \bw,\bv\in\Bcal.
\end{equation}
The following lemma bounds the interior and exterior brackets.

\begin{lemma}[Quadratic bounds for the normal entropy brackets]
\label{lem:LF-quadratic}
For every $\bw,\bv\in\Bcal$ and every unit vector $\bn$,
\begin{equation}\label{eq:LF-entropy-bracket-quadratic}
    \frac{\alpha\mu_\eta}{8}|\bw-\bv|^2
    \leq\frac12[\alpha\Reta(\bw,\bv)\pm R_{\Q_{\bn}}(\bw,\bv)]
    \leq\frac{3\alpha\nu_\eta}{8}|\bw-\bv|^2.
\end{equation}
\end{lemma}
\begin{proof}
With $\boldsymbol\delta=\bw-\bv$, the fundamental theorem of calculus
and normal entropy compatibility \eqref{eq:normal-entropy-pair} give
\begin{equation}\label{eq:RQ-integral}
    R_{\Q_{\bn}}(\bw,\bv)
       =\int_0^1[\eta'(\bv+\theta\boldsymbol\delta)-\eta'(\bv)]
          \bfv_{\bn}'(\bv+\theta\boldsymbol\delta)
          \boldsymbol\delta\dd\theta.
\end{equation}
Along this segment, the Hessian bound gives
$|\eta'(\bv+\theta\boldsymbol\delta)-\eta'(\bv)|
\leq\nu_\eta \theta|\boldsymbol\delta|$. Using
\eqref{eq:directional-flux-bounds}, \eqref{eq:Reta-bounds}, and
\eqref{eq:alpha-system}, we obtain
\begin{equation}\label{eq:RQ-bound}
    |R_{\Q_{\bn}}(\bw,\bv)|
    \leq\frac{\nu_\eta L_f}{2}|\bw-\bv|^2
    \leq\kappa_\eta L_f\Reta(\bw,\bv)
    \leq\frac\alpha2\Reta(\bw,\bv).
\end{equation}
Consequently,
$(\alpha/4)\Reta(\bw,\bv)
    \leq (1/2)[\alpha\Reta(\bw,\bv)\pm R_{\Q_{\bn}}(\bw,\bv)]
    \leq({3\alpha}/{4})\Reta(\bw,\bv).$
The two bounds in \eqref{eq:Reta-bounds} complete the proof.
\end{proof}

\subsubsection{Interior- and exterior-trace contributions}

Applying \eqref{eq:LF-entropy-bracket-quadratic} to
\eqref{eq:quadratic-interface-reference} separates the squared
deviations of the interior and exterior traces. Define the scaled
squared trace deviations and their RK-weighted averages by
\begin{equation}\label{eq:stage-quadratic-trace-sums}
    V_K^{(\ell),\pm}
       :=\frac{h_K}{|K|}
          \norm{\bu^{(\ell),\pm}-\overline{\bu}_K}_{\partial K}^2, \qquad  V_K^{\mathrm{RK},\pm}
       :=\sum_{\ell=1}^s b_\ell V_K^{(\ell),\pm}.
\end{equation}
The superscripts $-$ and $+$ indicate the interior and exterior trace
sides, respectively. Both traces are compared with the old cell average
$\overline{\bu}_K$. All these quantities are nonnegative by definition and Assumption~\ref{ass:rk-weights}.

\begin{lemma}[Quadratic estimate of the RK entropy-flux contribution]
\label{lem:entropy-flux-difference-quadratic}
Under the assumptions of Theorem~\ref{thm:weak-entropy}, except
for its time-step restriction,
\begin{equation}\label{eq:RK-relative-interface}
    \wh{Q}_K^{\mathrm{RK}}
       -\eta'(\overline{\bu}_K)\wh{\mathbf F}_K^{\mathrm{RK}}
    \leq\frac{3\alpha\nu_\eta}{8}V_K^{\mathrm{RK},-}
          -\frac{\alpha\mu_\eta}{8}V_K^{\mathrm{RK},+}.
\end{equation}
\end{lemma}
\begin{proof}
In \eqref{eq:LF-relative-quadratic-decomposition}, apply the upper bound
in \eqref{eq:LF-entropy-bracket-quadratic} to the interior bracket and
the lower bound to the subtracted exterior bracket. Integrating in
\eqref{eq:quadratic-interface-reference} yields
\begin{equation}\label{eq:stage-relative-interface}
    \wh{Q}_K^{(\ell)}
       -\eta'(\overline{\bu}_K)\wh{\mathbf F}_K^{(\ell)}
    \leq\frac{3\alpha\nu_\eta}{8}V_K^{(\ell),-}
          -\frac{\alpha\mu_\eta}{8}V_K^{(\ell),+}.
\end{equation}
Multiplying by $b_\ell\geq0$, summing, and using
\eqref{eq:flux-sums} and \eqref{eq:stage-quadratic-trace-sums} complete the proof. 
\end{proof}

\subsection{Term III: the squared numerical-flux sum}

By $\int_{\partial K}\bfv_{\bn_K}(\bv)\dd s=\mathbf0$,
\begin{equation}\label{eq:numerical-flux-reference}
    \wh{\mathbf F}_K^{(\ell)}
    =\frac{h_K}{|K|}\int_{\partial K}
       [\wh{\bfv}_{\bn_K}(\bu^{(\ell),-},\bu^{(\ell),+})
                      -\bfv_{\bn_K}(\overline{\bu}_K)]\dd s.
\end{equation}
We first bound the integrand.

\begin{lemma}[Squared normal numerical-flux deviation]
\label{lem:squared-flux-deviation}
For every $\bw_1,\bw_2,\bv\in\Bcal$ and every unit normal $\bn$,
\begin{equation}\label{eq:LF-physical-quadratic}
    |\wh{\bfv}_{\bn}(\bw_1,\bw_2)-\bfv_{\bn}(\bv)|^2
    \leq\frac{9\alpha^2}{8}
                   (|\bw_1-\bv|^2+|\bw_2-\bv|^2).
\end{equation}
\end{lemma}
\begin{proof}
Subtract $\bfv_{\bn}(\bv)$ in the Lax--Friedrichs formula and write
$\bw_2-\bw_1=(\bw_2-\bv)-(\bw_1-\bv)$. The bound
\eqref{eq:directional-flux-bounds} and convexity of $\Bcal$ imply
Lipschitz continuity of $\bfv_{\bn}$ with constant $L_f$. Hence
\begin{equation}\label{eq:LF-physical-Lipschitz}
    |\wh{\bfv}_{\bn}(\bw_1,\bw_2)-\bfv_{\bn}(\bv)|
    \leq\frac{\alpha+L_f}{2}(|\bw_1-\bv|+|\bw_2-\bv|)
    \leq\frac{3\alpha}{4}(|\bw_1-\bv|+|\bw_2-\bv|).
\end{equation}
Here $L_f\leq\alpha/(2\kappa_\eta)\leq\alpha/2$.
Squaring and applying $(a+b)^2\leq2(a^2+b^2)$ proves \eqref{eq:LF-physical-quadratic}.
\end{proof}

For any square-integrable $\mathbf g$ on
$\partial K$, the Cauchy--Schwarz inequality gives
\begin{equation}\label{eq:boundary-CS}
    \left|\frac{h_K}{|K|}\int_{\partial K}\mathbf g\dd s\right|^2
    \leq\frac{h_K|\partial K|}{|K|}
          \frac{h_K}{|K|}\norm{\mathbf g}_{\partial K}^2
    \leq C_k\frac{h_K}{|K|}\norm{\mathbf g}_{\partial K}^2.
\end{equation}
The last inequality follows from \eqref{eq:cell-surface-ratio}.

\begin{lemma}[Squared RK-averaged numerical-flux sum]
\label{lem:numerical-flux-sum-quadratic}
Under the state, viscosity, and RK-weight assumptions,
\begin{equation}\label{eq:RK-numerical-flux-sum}
    |\wh{\mathbf F}_K^{\mathrm{RK}}|^2
    \leq\frac{9\alpha^2C_k}{8}
           (V_K^{\mathrm{RK},-}+V_K^{\mathrm{RK},+}).
\end{equation}
\end{lemma}
\begin{proof}
Apply \eqref{eq:boundary-CS} to \eqref{eq:numerical-flux-reference}, followed by
Lemma~\ref{lem:squared-flux-deviation} pointwise on each face:
\begin{equation}\label{eq:stage-numerical-flux-sum}
    |\wh{\mathbf F}_K^{(\ell)}|^2
    \leq\frac{9\alpha^2C_k}{8}
                  (V_K^{(\ell),-}+V_K^{(\ell),+}).
\end{equation}
By weighted Cauchy--Schwarz and Assumption~\ref{ass:rk-weights},
\begin{equation}\label{eq:RK-Jensen-flux}
\begin{aligned}
    |\wh{\mathbf F}_K^{\mathrm{RK}}|^2
    &=\left|\sum_{\ell=1}^s b_\ell\wh{\mathbf F}_K^{(\ell)}\right|^2 \leq\left(\sum_{\ell=1}^s b_\ell\right)
            \sum_{\ell=1}^s b_\ell|\wh{\mathbf F}_K^{(\ell)}|^2
     =\sum_{\ell=1}^s b_\ell|\wh{\mathbf F}_K^{(\ell)}|^2.
\end{aligned}
\end{equation}
Insert \eqref{eq:stage-numerical-flux-sum} and use
\eqref{eq:stage-quadratic-trace-sums} to obtain \eqref{eq:RK-numerical-flux-sum}.
\end{proof}

It remains to control $V_K^{\mathrm{RK},-}$ by the Jensen gap; the
exterior-trace contribution is nonpositive under \eqref{eq:higher-order-CFL}.

\subsection{Control of the local stages by the Jensen gap}

\begin{lemma}[Local operator relative to a constant state]
\label{lem:local-operator-general-k}
Let $\bw\in[\mathbb Z^k(K)]^m$ take values in $\Bcal$.
For every constant vector $\bv\in\mathbb R^m$,
\begin{equation}\label{eq:local-operator-general-k}
    \norm{\loc(\bw)}_K
    \leq\frac{C_kL_f}{h_K}\norm{\bw-\bv}_K.
\end{equation}
\end{lemma}

\begin{proof}
Let $\Pi_K$ be the $L^2$ projection onto
$[\mathbb Z^k(K)]^m$. Integration by parts in \eqref{eq:local-operator} gives
$\loc(\bw)=-\Pi_K\!\left(
       \sum_{r=1}^d\bfv_r'(\bw)\partial_{x_r}\bw\right).$
By $L^2$-contractivity, \eqref{eq:directional-flux-bounds}, and
\eqref{eq:inverse-trace},
\begin{equation}\label{eq:local-operator-bound-proof}
    \norm{\loc(\bw)}_K
       \leq L_f\norm{\nabla\bw}_K
       =L_f\norm{\nabla(\bw-\bv)}_K
       \leq\frac{C_kL_f}{h_K}\norm{\bw-\bv}_K.
\end{equation}
\end{proof}

\begin{lemma}[Uniform control of local stages]
\label{lem:all-stages}
Under the assumptions of Theorem~\ref{thm:weak-entropy}, except
for its time-step restriction, suppose $C_ka_*\sigma_K\leq1$.
Then
\begin{equation}\label{eq:all-stage-L2}
    \norm{\bu^{(\ell)}-\overline{\bu}_K}_K
       \leq2\norm{\bu-\overline{\bu}_K}_K,
       \qquad \ell=1,\ldots,s.
\end{equation}
\end{lemma}
\begin{proof}
The estimate holds for $\bu^{(1)}=\bu$. Suppose it holds for all $j<\ell$.
The recurrence \eqref{eq:general-crkdg-b} and
Lemma~\ref{lem:local-operator-general-k} give
\begin{equation}\label{eq:stage-recursion-first}
\begin{aligned}
    \norm{\bu^{(\ell)}-\overline{\bu}_K}_K
    &\leq\norm{\bu-\overline{\bu}_K}_K
       +C_kL_f\lambda_K\sum_{j<\ell}|a_{\ell j}|
          \norm{\bu^{(j)}-\overline{\bu}_K}_K\\
    &\leq(1+2C_ka_*L_f\lambda_K)\norm{\bu-\overline{\bu}_K}_K
     \leq2\norm{\bu-\overline{\bu}_K}_K,
\end{aligned}
\end{equation}
since $2L_f\leq\alpha$ and $C_ka_*\sigma_K\leq1$.
\end{proof}

\begin{corollary}[The Jensen gap controls the stage traces]
\label{cor:all-stage-traces}
Under the hypotheses of Lemma~\ref{lem:all-stages},
\begin{align}
    V_K^{(\ell),-}
       &\leq\frac{4C_k}{|K|}\norm{\bu-\overline{\bu}_K}_K^2
        \leq\frac{8C_k}{\mu_\eta}J_K(\bu),
        \qquad \ell=1,\ldots,s,\label{eq:all-stage-traces}\\
    V_K^{\mathrm{RK},-}&\leq\frac{8C_k}{\mu_\eta}J_K(\bu).
        \label{eq:RK-interior-trace-J}
\end{align}
\end{corollary}
\begin{proof}
Apply the trace estimate \eqref{eq:inverse-trace},
Lemma~\ref{lem:all-stages}, and the Jensen-gap bound \eqref{eq:J-bounds}.
The RK estimate follows because $V_K^{\mathrm{RK},-}$ is a convex
combination of the $V_K^{(\ell),-}$.
\end{proof}

\subsection{Completion of the entropy estimate}

\begin{proof}[Proof of Theorem~\ref{thm:weak-entropy}]
The CFL condition implies $C_ka_*\sigma_K\leq1$, so
Corollary~\ref{cor:all-stage-traces} applies. Substituting
\eqref{eq:RK-relative-interface}, \eqref{eq:RK-numerical-flux-sum}, and
\eqref{eq:RK-interior-trace-J} into \eqref{eq:master-decomposition-expanded} gives
\begin{equation}\label{eq:higher-order-defect-J}
\begin{aligned}
 &\eta(\overline{\bu}_K^{n+1,\star})
    -\fint_K\eta(\bu)\dd\bx
    +\lambda_K\wh{Q}_K^{\mathrm{RK}}\\
 &\quad\leq
    -\left[1-3\kappa_\eta C_k\sigma_K
                     \left(1+\frac32C_k\sigma_K\right)\right]J_K(\bu)-\frac{\mu_\eta\sigma_K}{8}
          \left(1-\frac92\kappa_\eta C_k\sigma_K\right)
                                     V_K^{\mathrm{RK},+}.
\end{aligned}
\end{equation}
Since $\kappa_\eta C_k\sigma_K\leq1/9$ and $\kappa_\eta\geq1$,
\begin{equation}\label{eq:higher-order-coefficient-bounds}
\begin{aligned}
 3\kappa_\eta C_k\sigma_K
          \left(1+\frac32C_k\sigma_K\right)
   \leq\frac13\left(1+\frac1{6\kappa_\eta}\right)
    \leq\frac7{18},\quad 
 1-\frac92\kappa_\eta C_k\sigma_K&\geq\frac12.
\end{aligned}
\end{equation}
The nonnegativity of $J_K(\bu)$ and $V_K^{\mathrm{RK},+}$ now yields
\begin{equation}\label{eq:higher-order-final-residual}
\begin{aligned}
 &\eta(\overline{\bu}_K^{n+1,\star})
    -\fint_K\eta(\bu)\dd\bx
    +\lambda_K\wh{Q}_K^{\mathrm{RK}}\leq-\frac{11}{18}J_K(\bu)
           -\frac{\mu_\eta\sigma_K}{16}V_K^{\mathrm{RK},+}
    \leq0.
\end{aligned}
\end{equation}
Using \eqref{eq:entropy-flux-sums} proves \eqref{eq:higher-order-weak}.
\end{proof}

\section{Strong multi-entropy stability}
\label{sec:entropy-limiter}

We next extend the weak cell entropy inequality to a prescribed family of
entropies and use a cell-average-preserving scaling to obtain a fully
discrete cRKDG scheme satisfying the corresponding strong cell entropy
inequalities simultaneously.

\subsection{Weak multi-entropy stability}
\label{subsec:weak-multi-entropy}

Denote by $\bu_h^n$ the solution at time level $t^n$ and by
$\bu_h^{n+1,\star}$ the pre-limiter solution in
\eqref{eq:general-crkdg-c}.

Let $\mathcal I$ be a fixed index set, finite or infinite,
and consider the entropy family
\begin{equation}\label{eq:entropy-family}
    \left\{(\eta_q,\boldsymbol{\mathcal Q}_q)\right\}_{q\in\mathcal I},
    \qquad
    \boldsymbol{\mathcal Q}_q=(\Q_{q,1},\ldots,\Q_{q,d}).
\end{equation}

\begin{assumption}[A uniformly conditioned entropy family]
\label{ass:multi-entropies}
For every $q\in\mathcal I$, the pair $(\eta_q,\boldsymbol{\mathcal Q}_q)$ is
$C^2$ on an open neighborhood of $\Bcal$ and satisfies
\begin{equation}\label{eq:multi-entropy-compatibility}
    \Q_{q,r}'(\bv)=\eta_q'(\bv)\bfv_r'(\bv),
    \qquad \bv\in\Bcal,\quad r=1,\ldots,d.
\end{equation}
There are constants $0<\mu_{\eta_q}\leq\nu_{\eta_q}<\infty$ such that
\begin{equation}\label{eq:multi-entropy-convexity}
    \mu_{\eta_q}I_m\preceq\eta_q''(\bv)\preceq\nu_{\eta_q}I_m,
    \qquad \bv\in\Bcal.
\end{equation}
The curvature ratios satisfy the uniform bound
\begin{equation}\label{eq:multi-entropy-conditioning}
    \kappa_{\eta_q}:=\frac{\nu_{\eta_q}}{\mu_{\eta_q}},
    \qquad
    \kappa_*:=\sup_{q\in\mathcal I}\kappa_{\eta_q}<\infty.
\end{equation}
\end{assumption}
The curvature constants may depend on $q$; only their ratios must be
uniformly bounded. For a finite family, the supremum in
\eqref{eq:multi-entropy-conditioning} is a maximum.

Impose the common sufficient conditions
\begin{equation}\label{eq:multi-entropy-CFL}
    \alpha\geq2\kappa_*L_f,
    \qquad
    \sigma_K\leq{(C_k\max\{a_*,9\kappa_*\})^{-1}},
    \qquad K\in\mathcal T_h.
\end{equation}
For $q\in\mathcal I$, set $\Q_{q,\bn}:=\sum_{r=1}^d n_r\Q_{q,r}$
and define $\wh{\Q}_{q,\bn}$ and
$\wh{\Q}_{q,\bn_K}^{n,\mathrm{RK}}$ by
\eqref{eq:lf-entropy-flux} and \eqref{eq:rk-entropy-flux-general},
with $(\eta,\Q_r)$ replaced by $(\eta_q,\Q_{q,r})$.
All entropy fluxes use the same $\alpha$ and the same local stages.

\begin{lemma}[Weak multi-entropy stability]
\label{lem:weak-multi-entropy}
Under Assumptions~\ref{ass:state-range}, \ref{ass:rk-weights},
and~\ref{ass:multi-entropies}, suppose
\eqref{eq:multi-entropy-CFL} holds. Then, for every
$K\in\mathcal T_h$ and $q\in\mathcal I$,
\begin{equation}\label{eq:weak-budget-form}
    \eta_q\!\left(\overline{\bu}_{h,K}^{n+1,\star}\right)
    \leq\fint_K\eta_q(\bu_h^n)\dd\bx
      -\frac{\tau}{|K|}\int_{\partial K}
                    \wh{\Q}_{q,\bn_K}^{n,\mathrm{RK}}\dd s.
\end{equation}
\end{lemma}
\begin{proof}
For each $q\in\mathcal I$, the bound
$\kappa_{\eta_q}\leq\kappa_*$ ensures that
\eqref{eq:multi-entropy-CFL} implies the hypotheses of
Theorem~\ref{thm:weak-entropy} for that entropy pair.
The theorem then gives \eqref{eq:weak-budget-form}.
\end{proof}

\subsection{Entropy correction by cellwise scaling}

We apply the explicit cell-average-preserving scaling in \cite{LiuGuoJiangSun,LiuJiangSun2026} to the family
in \eqref{eq:entropy-family}. Only the end-of-step solution is modified;
the local stages and their numerical fluxes remain unchanged.

For each $q\in\mathcal I$, define the entropy budget
\begin{equation}\label{eq:entropy-budget}
    B_{K,q}^n:=\fint_K\eta_q(\bu_h^n)\dd\bx
       -\frac{\tau}{|K|}\int_{\partial K}
                    \wh{\Q}_{q,\bn_K}^{n,\mathrm{RK}}\dd s.
\end{equation}
For $0\leq\theta\leq1$, introduce the common scaling path and its
entropy profiles:
\begin{align}
    \bu_{h,K}^{n,\theta}(\bx)
    :=&\overline{\bu}_{h,K}^{n+1,\star}
       +\theta\left(\bu_h^{n+1,\star}(\bx)
                      -\overline{\bu}_{h,K}^{n+1,\star}\right),
    \qquad \bx\in\overline K,\label{eq:entropy-scaling-path}\\
    \Psi_{K,q}(\theta):=&\fint_K
                   \eta_q(\bu_{h,K}^{n,\theta})\dd\bx.\label{eq:Psi-def}
\end{align}
By construction, the path preserves the cell average $\overline{\bu}_{h,K}^{n,\theta} = \overline{\bu}_{h,K}^{n+1,\star}$. Assumption~\ref{ass:state-range} and the convexity of $\Bcal$
ensure that all scaled states remain in $\Bcal$. Lemma~\ref{lem:weak-multi-entropy}, convexity of $\eta_q$,
and Jensen's inequality give
\begin{align}
    \Psi_{K,q}(0)&=\eta_q\!\left(\overline{\bu}_{h,K}^{n+1,\star}\right)
       \leq B_{K,q}^n,\label{eq:Psi-endpoints}\\
    \Psi_{K,q}(0)&\leq\Psi_{K,q}(\theta)
       \leq\Psi_{K,q}(0)+\theta[\Psi_{K,q}(1)-\Psi_{K,q}(0)].
       \label{eq:Psi-profile-bounds}
\end{align}
Set
\begin{equation}\label{eq:explicit-theta}
    \theta_{K,q}:=
    \begin{cases}
       1, & \Psi_{K,q}(1)\leq B_{K,q}^n,\\[0.5em]
       \displaystyle
       \frac{B_{K,q}^n-\Psi_{K,q}(0)}
            {\Psi_{K,q}(1)-\Psi_{K,q}(0)},
          & \Psi_{K,q}(1)>B_{K,q}^n,
    \end{cases}
\end{equation}
and use a single coefficient for all entropies:
\begin{equation}\label{eq:multi-entropy-theta}
    \theta_K:=\inf_{q\in\mathcal I}\theta_{K,q}.
\end{equation}
The active branch has a positive denominator and a numerator between
zero and that denominator. Thus $0\leq\theta_K\leq\theta_{K,q}\leq1$;
if $\Psi_{K,q}(1)=\Psi_{K,q}(0)$, the inactive branch applies. For finite
$\mathcal I$, \eqref{eq:multi-entropy-theta} is a minimum; for infinite
$\mathcal I$, the infimum need not be attained.

\begin{remark}[Round-off and limiter activation]\label{rem:roundoff}
In floating-point arithmetic, the computed entropy budget $B_{K,q}^n$
and entropy values $\Psi_{K,q}$ have errors of size
$\mathcal{O}(\epsilon S_{K,q})$, where $\epsilon$ is the machine
precision and $S_{K,q}$ is a representative scale of the quantities
entering the budget. Hence a round-off-sized excess
$\Psi_{K,q}(1)-B_{K,q}^n$ may spuriously activate the limiter. Near a
constant state, this can produce an essentially arbitrary
$\theta_{K,q}$ because both the numerator and denominator in
\eqref{eq:explicit-theta} are at round-off level; for a smooth
well-resolved solution, repeated round-off-sized corrections may also
affect accuracy on sufficiently fine meshes. We therefore activate the
limiter only when
$\Psi_{K,q}(1)-B_{K,q}^n>100\epsilon S_{K,q}$ 
and set $\theta_{K,q}=1$ otherwise. As a safeguard, we replace $\theta_K$ with
$\max\{0,\min\{1,\theta_K\}\}$.
\end{remark}
\subsection{Fully discrete multi-entropy-stable cRKDG scheme}
Given $\bu_h^n$, first compute the local stages and the pre-limiter
solution $\bu_h^{n+1,\star}$ from \eqref{eq:general-crkdg}. For each cell
$K$, compute the entropy budgets \eqref{eq:entropy-budget} and the common
coefficient $\theta_K$ from
\eqref{eq:explicit-theta}--\eqref{eq:multi-entropy-theta}. The new solution is
\begin{equation}\label{eq:post-limiter-solution}
    \bu_h^{n+1}|_K
    :=\overline{\bu}_{h,K}^{n+1,\star}
      +\theta_K\left(\bu_h^{n+1,\star}
        -\overline{\bu}_{h,K}^{n+1,\star}\right).
\end{equation}
We refer to \eqref{eq:general-crkdg} together with
\eqref{eq:entropy-budget}, \eqref{eq:explicit-theta}--\eqref{eq:post-limiter-solution}
as the fully discrete multi-entropy-stable cRKDG scheme.

\begin{theorem}[Strong multi-entropy stability]
\label{thm:strong-cell-entropy}
Under Assumptions~\ref{ass:state-range}, \ref{ass:rk-weights},
and~\ref{ass:multi-entropies}, suppose
\eqref{eq:multi-entropy-CFL} holds. The solution produced by the fully
discrete multi-entropy-stable cRKDG scheme belongs to $\mathbb V_h^k$ and satisfies, for every
$K\in\mathcal T_h$ and $q\in\mathcal I$,
\begin{equation}\label{eq:strong-cell-entropy}
    \fint_K\eta_q(\bu_h^{n+1})\dd\bx
    \leq\fint_K\eta_q(\bu_h^n)\dd\bx
       -\frac{\tau}{|K|}\int_{\partial K}
                    \wh{\Q}_{q,\bn_K}^{n,\mathrm{RK}}\dd s.
\end{equation}
Moreover, the conservative cell-average update is unchanged:
\begin{equation}\label{eq:post-limiter-average}
    \overline{\bu}_{h,K}^{\,n+1}
    =\overline{\bu}_{h,K}^{n+1,\star}
    =\overline{\bu}_{h,K}^{\,n}
       -\frac{\tau}{|K|}\int_{\partial K}
                    \wh{\bfv}_{\bn_K}^{n,\mathrm{RK}}\dd s.
\end{equation}
\end{theorem}
\begin{proof}
The cellwise scaling preserves the finite element space and the cell average.
For any $q\in\mathcal I$, \eqref{eq:explicit-theta} gives
$\Psi_{K,q}(0)+\theta_{K,q}[\Psi_{K,q}(1)-\Psi_{K,q}(0)]\leq B_{K,q}^n$
in both branches. Since $\theta_K\leq\theta_{K,q}$ and the bracket is
nonnegative, \eqref{eq:Psi-profile-bounds} implies
\begin{equation}\label{eq:common-entropy-budget}
\begin{aligned}
    \fint_K\eta_q(\bu_h^{n+1})\dd\bx
    &=\Psi_{K,q}(\theta_K) \leq\Psi_{K,q}(0)+\theta_K[\Psi_{K,q}(1)-\Psi_{K,q}(0)]
     \leq B_{K,q}^n.
\end{aligned}
\end{equation}
Substituting \eqref{eq:entropy-budget} proves
\eqref{eq:strong-cell-entropy} for every $q$, whether or not the infimum
in \eqref{eq:multi-entropy-theta} is attained.
\end{proof}

\begin{corollary}[Global multi-entropy stability]
\label{cor:global-entropy}
Under the hypotheses of Theorem~\ref{thm:strong-cell-entropy}, assume periodic boundary conditions or boundary conditions with nonnegative total outward numerical entropy flux for each $q$.
\begin{equation}\label{eq:global-entropy}
    \sum_{K\in\mathcal T_h}\int_K\eta_q(\bu_h^{n+1})\dd\bx
    \leq\sum_{K\in\mathcal T_h}\int_K\eta_q(\bu_h^n)\dd\bx,
    \qquad q\in\mathcal I.
\end{equation}
\end{corollary}
\begin{proof}
Multiply \eqref{eq:strong-cell-entropy} by $|K|$ and sum over the mesh.
Interior face terms cancel by \eqref{eq:normal-entropy-flux-conservation};
the total outward boundary flux is nonnegative under the stated conditions.
\end{proof}

Any coefficient in $[0,\theta_K]$ also satisfies all entropy budgets.
For an infinite family, implementing the construction requires
evaluating the infimum or obtaining a certified lower bound for it.
The result concerns the prescribed entropies and does not establish
an accuracy estimate when the limiter is active. 

\subsection{Lax--Wendroff convergence}
\label{sec:lax-wendroff}

The strong cell entropy inequalities above do more than provide a discrete
stability property: they also identify the entropy admissibility of any
strongly convergent numerical limit. We thus conclude this section
with a Lax--Wendroff-type convergence theorem. Its proof, deferred to
Appendix~\ref{app:lax-wendroff}, combines the conservative cell-average
update \eqref{eq:post-limiter-average}, the strong cell entropy inequalities
\eqref{eq:strong-cell-entropy}, and the local-stage estimate
\eqref{eq:all-stage-L2} for the fixed entropy family
\eqref{eq:entropy-family}.

Consider a sequence of meshes satisfying the assumptions of
Subsection~\ref{subsec:dg-space}, with
$h:=\max_{K\in\mathcal T_h}h_K\to0$, and assume additionally that the
mesh family is quasi-uniform. For each mesh, let $t^n=n\tau_h$, where
$\tau_h=T/N_h$, and define
\begin{equation}\label{eq:LW-reconstruction}
    \bu_h(\bx,t):=\bu_h^n(\bx),
    \qquad t\in[t^n,t^{n+1}),\quad n=0,\ldots,N_h-1.
\end{equation}

\begin{theorem}[Lax--Wendroff convergence with multiple entropies]
\label{thm:lax-wendroff}
Fix the polynomial degree, RK coefficients, parameter $\alpha$, state
set $\Bcal$, and entropy family indexed by $\mathcal I$, independently
of $h$. Suppose that the hypotheses of
Theorem~\ref{thm:strong-cell-entropy}, including
\eqref{eq:multi-entropy-CFL}, hold at every time step on the meshes
specified above. Assume
\begin{equation}\label{eq:LW-strong-convergence}
    \bu_h^0\longrightarrow\bu_0
        \quad\text{in }L^1(\Omega),
    \qquad
    \bu_h\longrightarrow\bu
        \quad\text{in }L^1(\Omega_T),
    \quad \Omega_T:=\Omega\times(0,T).
\end{equation}
Then $\bu$ takes values in $\Bcal$ almost everywhere and is a weak
solution of \eqref{eq:system} in $\Omega_T$ with initial data $\bu_0$:
for every scalar test function
$\varphi\in C_c^\infty(\Omega\times[0,T))$,
\begin{equation}\label{eq:LW-weak-solution}
\begin{aligned}
    &\int_0^T\int_\Omega
       \left[\bu\,\partial_t\varphi
          +\sum_{r=1}^d\bfv_r(\bu)\,\partial_{x_r}\varphi\right]
            \dd\bx\dd t
       +\int_\Omega\bu_0\varphi(\bx,0)\dd\bx=\mathbf0.
\end{aligned}
\end{equation}
Moreover, for every $q\in\mathcal I$ and every such test function
$\varphi\geq0$,
\begin{equation}\label{eq:LW-entropy-solution}
\begin{aligned}
    &\int_0^T\int_\Omega
       \left[\eta_q(\bu)\,\partial_t\varphi
          +\boldsymbol{\mathcal Q}_q(\bu)\cdot\nabla\varphi\right]
            \dd\bx\dd t
       +\int_\Omega\eta_q(\bu_0)\varphi(\bx,0)\dd\bx\geq0.
\end{aligned}
\end{equation}
\end{theorem}

The proof first establishes convergence of the local stages.
A weighted face estimate then gives consistency of the RK-averaged
numerical physical and entropy fluxes. Finally, we pass to the
limit in the conservative update and the strong cell entropy
inequalities; see Appendix~\ref{app:lax-wendroff}.

\section{Numerical experiments}
\label{sec:numerical-experiments}

We use the following tests to assess high-order accuracy for smooth
solutions, simultaneous control of several entropy inequalities,
and robustness for shock problems. We also compare end-of-step
and stagewise entropy corrections in a smooth Burgers test. We denote the cRKDG scheme without the entropy correction by
Base and the fully discrete scheme of Section~\ref{sec:entropy-limiter}
by ES.

\paragraph{Discretization and quadrature}
We use uniform meshes and a Legendre basis for
$\mathbb Z^k=\mathbb P^k$. For cRKDG$p$ with $p=k+1$, we use
explicit midpoint for $p=2$, Heun's third-order method for $p=3$,
and the classical fourth-order RK method for $p=4$ \cite{ChenSunXing2024}. The smooth convergence tests use $p=2,3,4$; the remaining
tests use cRKDG4. The time step is $\tau=c \Delta x/\alpha$, with $c=0.01$ in
the accuracy tests, $c=0.08$ in the other 1D tests, and $c=0.04$ in the other 2D tests.
For linear advection, Burgers' equation, the 1D Buckley--Leverett
equation, and the 2D Buckley--Leverett equation, we use $\alpha=1,2,3,3.476$, respectively. For systems,
$\alpha^n=1.05\,\lambda(\bu_h^n)$ bounds the characteristic speeds at
quadrature points and in the boundary data and is fixed during each
attempted time step; if a stage produces a larger speed, the step is
restarted with the enlarged $\alpha$. In 1D, cell integrals are evaluated
with the $(2k+1)$-point Gauss--Legendre rule. In 2D, cell and face
integrals use the $(k+1)\times(k+1)$ tensor-product and $(k+1)$-point
Gauss--Legendre rules, respectively. These positive-weight rules satisfy
the degree-$2k$ exactness requirement of
Remark~\ref{rem:quadrature-entropy}.

\paragraph{Entropy correction, admissibility, and diagnostics}
ES applies \eqref{eq:explicit-theta}--\eqref{eq:post-limiter-solution}
once per completed step, with the round-off safeguard of
Remark~\ref{rem:roundoff}. Each attempted
step is checked a posteriori: it is recomputed with $\tau/2$ if the weak budget \eqref{eq:weak-budget-form} or the strong
inequality \eqref{eq:strong-cell-entropy} fails by more than
$100\epsilon S_{K,q}$ in any cell, or if an inadmissible state occurs.
Thus the experiments verify entropy enforcement under practical time
steps; they are not intended to measure the sharpness of the sufficient CFL bound. For the blast-wave and double Mach tests, the
bound-preserving limiter of \cite{LiuSunZhang2025} scales the stages and
$\bu_h^{n+1,\star}$ so that density and pressure are at least $10^{-13}$
at all quadrature and check points; for the 2D Buckley--Leverett test it
enforces $u\in[0,1]$. This stage scaling modifies the uncorrected cRKDG
stage recursion used in Theorem~\ref{thm:weak-entropy}; accordingly, in
these bound-preserving tests we rely on the a posteriori weak-budget
check above rather than claiming that Theorem \ref{thm:weak-entropy} applies unchanged. The subsequent entropy scaling is a convex combination with
the cell average and therefore preserves the imposed bounds. No
additional shock-capturing limiter is used. Solution plots show cell
averages. Whenever a reference solution is shown, it is computed on a
substantially finer mesh with the WENO method. 

\subsection{One-dimensional scalar problems}

\begin{example}[Linear advection]\label{ex:linear-advection}
We solve $u_t+u_x=0$ on $[0,2\pi]$ with periodic boundary conditions,
$u(x,0)=\sin^4x$, and $T=1$. ES enforces
$(\eta_1,\Q_1)=(e^u,e^u)$ and $(\eta_2,\Q_2)=(u^2/2,u^2/2)$
simultaneously. Table~\ref{tab:num1d-linear} shows the designed orders
$p=2,3,4$ for both Base and ES, with comparable errors.
Figure~\ref{fig:num1d-linear} shows that the active cells form narrow
bands transported with the smooth solution.
\end{example}

\begin{table}[!htbp]
\centering
\caption{Example~\ref{ex:linear-advection}, $T=1$:
$L^2$ errors. ES:
$\eta_1=e^u$ and $\eta_2=u^2/2$.}
\label{tab:num1d-linear}
\normalsize
\begin{tabular}{c r cc cc}
\toprule
& & \multicolumn{2}{c}{Base} & \multicolumn{2}{c}{ES}\\
\cmidrule(lr){3-4}\cmidrule(lr){5-6}
$p$ & $N$ & Error & Order & Error & Order\\
\midrule
\multirow{5}{*}{$2$}
& 20  & $3.380\times10^{-2}$ & --   & $3.399\times10^{-2}$ & --   \\
& 40  & $7.761\times10^{-3}$ & 2.12 & $7.785\times10^{-3}$ & 2.13 \\
& 80  & $1.875\times10^{-3}$ & 2.05 & $1.879\times10^{-3}$ & 2.05 \\
& 160 & $4.640\times10^{-4}$ & 2.01 & $4.649\times10^{-4}$ & 2.01 \\
\midrule
\multirow{5}{*}{$3$}
& 20  & $2.358\times10^{-3}$ & --   & $2.358\times10^{-3}$ & --   \\
& 40  & $2.924\times10^{-4}$ & 3.01 & $2.923\times10^{-4}$ & 3.01 \\
& 80  & $3.651\times10^{-5}$ & 3.00 & $3.649\times10^{-5}$ & 3.00 \\
& 160 & $4.563\times10^{-6}$ & 3.00 & $4.559\times10^{-6}$ & 3.00 \\
\midrule
\multirow{5}{*}{$4$}
& 20  & $1.632\times10^{-4}$ & --   & $1.633\times10^{-4}$ & --   \\
& 40  & $1.042\times10^{-5}$ & 3.97 & $1.048\times10^{-5}$ & 3.96 \\
& 80  & $6.446\times10^{-7}$ & 4.01 & $6.536\times10^{-7}$ & 4.00 \\
& 160 & $4.035\times10^{-8}$ & 4.00 & $4.091\times10^{-8}$ & 4.00 \\
\bottomrule
\end{tabular}
\end{table}

\begin{figure}[h!]\centering
\begin{subfigure}[t]{.325\linewidth}\centering
\includegraphics[width=\linewidth]{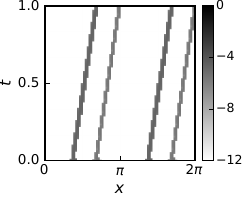}\caption{$p=2$}\label{fig:num1d-linear-p2}
\end{subfigure}\hfill
\begin{subfigure}[t]{.325\linewidth}\centering
\includegraphics[width=\linewidth]{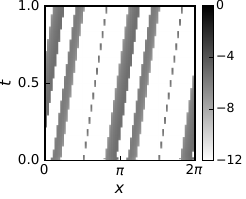}\caption{$p=3$}\label{fig:num1d-linear-p3}
\end{subfigure}\hfill
\begin{subfigure}[t]{.325\linewidth}\centering
\includegraphics[width=\linewidth]{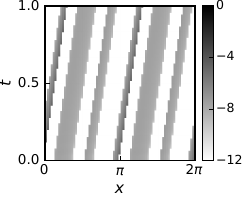}\caption{$p=4$}\label{fig:num1d-linear-p4}
\end{subfigure}
\caption{Example~\ref{ex:linear-advection}, ES with $\eta_1=e^u$ and $\eta_2=u^2/2$: limiter-activation
history. $\log_{10}(1-\theta_K)$ is plotted.}
\label{fig:num1d-linear}
\end{figure}

\begin{example}[Burgers' equation]\label{ex:burgers}
We solve $u_t+(u^2/2)_x=0$ on $[0,2\pi]$ with periodic boundary conditions,
$u(x,0)=0.5+\sin x$, and $(\eta,\Q)=(e^u,(u-1)e^u)$; the exact solution is
smooth until $t=1$. At $T=0.3$, we compare cRKDG4 with SSPRK(10,4)--DG,
i.e., SSPRK(10,4) \cite{Ketcheson2008} with $\DG$ at every stage. Its ES
variant performs a correction after each of the ten forward-Euler
substeps. Table~\ref{tab:num1d-burgers} shows fourth-order accuracy for cRKDG4 with a single
end-of-step correction, whereas the tested SSPRK(10,4)--DG variant
with stagewise correction exhibits first-order accuracy.
\end{example}

\begin{table}[!htbp]
\centering
\caption{Example~\ref{ex:burgers}, $T=0.3$:
$L^2$ errors. For cRKDG4, ES is applied once to
$u_h^{n+1,\star}$; for SSPRK(10,4)--DG, ES is applied after every
forward-Euler substep.}
\label{tab:num1d-burgers}
\normalsize
\begin{tabular}{l r cc cc}
\toprule
& & \multicolumn{2}{c}{Base} & \multicolumn{2}{c}{ES}\\
\cmidrule(lr){3-4}\cmidrule(lr){5-6}
Method & $N$ & Error & Order & Error & Order\\
\midrule
\multirow{5}{*}{cRKDG4}
& 20  & $1.761\times10^{-5}$  & --   & $1.762\times10^{-5}$  & --   \\
& 40  & $1.066\times10^{-6}$  & 4.05 & $1.066\times10^{-6}$  & 4.05 \\
& 80  & $6.600\times10^{-8}$  & 4.01 & $6.602\times10^{-8}$  & 4.01 \\
& 160 & $4.125\times10^{-9}$  & 4.00 & $4.125\times10^{-9}$  & 4.00 \\
\midrule
\multirow{5}{*}{SSPRK(10,4)--DG}
& 20  & $1.771\times10^{-5}$  & --   & $2.771\times10^{-4}$ & --   \\
& 40  & $1.074\times10^{-6}$  & 4.04 & $1.333\times10^{-4}$ & 1.06 \\
& 80  & $6.658\times10^{-8}$  & 4.01 & $6.450\times10^{-5}$ & 1.05 \\
& 160 & $4.163\times10^{-9}$  & 4.00 & $3.166\times10^{-5}$ & 1.03 \\
\bottomrule
\end{tabular}
\end{table}

We next compute to $T=1.08$, after shock formation, with $N=80$. Base exhibits global entropy increases at some time steps, whereas ES maintains nonincreasing global entropy up to the prescribed tolerance; see Figure~\ref{fig:num1d-burgers-entropy}. In the final ES step, the entropy residuals for the cell averages, the pre-limiter solution, and the post-limiter solution are given in Figure~\ref{fig:num1d-burgers-residuals}.

\begin{figure}[h!]\centering
\begin{subfigure}[t]{.4\linewidth}\centering
\includegraphics[width=\linewidth]{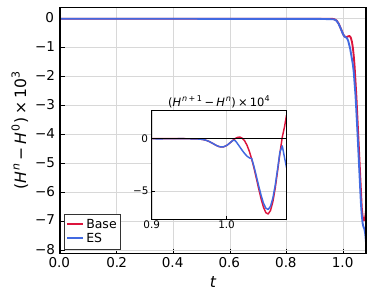}
\caption{total entropy}
\label{fig:num1d-burgers-entropy}
\end{subfigure}\qquad\qquad 
\begin{subfigure}[t]{.4\linewidth}\centering
\includegraphics[width=\linewidth]{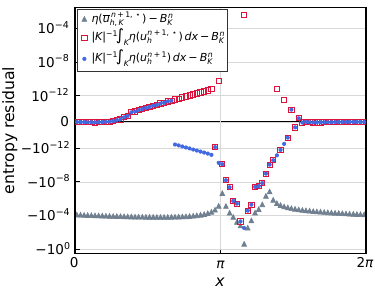}
\caption{residuals in the final ES step}
\label{fig:num1d-burgers-residuals}
\end{subfigure}
\caption{Example~\ref{ex:burgers}, $p=4$, $N=80$, $T=1.08$. Figure~\ref{fig:num1d-burgers-entropy} shows $H^n-H^0$, where $H^n:=\int_\Omega\eta(\bu_h^n)\dd\bx$ is evaluated with the $(2k+1)$-point Gauss--Legendre quadrature rule; the inset shows $H^{n+1}-H^n$ for $t\geq0.9$. Figure~\ref{fig:num1d-burgers-residuals} shows the weak, pre-correction, and post-correction residuals in the final ES step (symmetric logarithmic scale, linear for $|\cdot|<10^{-14}$).}
\label{fig:num1d-burgers}
\end{figure}

\begin{example}[Buckley--Leverett equation]\label{ex:buckley-leverett-1d}
We solve the Buckley--Leverett equation with flux
$f(u)=4u^2/(4u^2+(1-u)^2)$ on $[-0.5,0.5]$.
The initial data have a jump at $x=0$, and the boundary states
are fixed. The Riemann data are $(u_L,u_R)=(-3,3)$ (IC1) and
$(2,-2)$ (IC2); $N=80$ and $T=1$, and no bound-preserving limiter is
used. We enforce one of
$\eta_1(u)=u^2/2$, $\eta_2(u)=\Phi(u)$, and $\eta_3(u)=\Phi(u-1)$, or the
pair $\{\eta_2,\eta_3\}$, where
\[
\Phi(z)=z\arctan(20z)-\frac1{40}\log(1+400z^2),
\qquad \Phi'(z)=\arctan(20z),
\]
and $\Q_q'=\eta_q'f'$. Figure~\ref{fig:num1d-buckley} shows that ES-$\eta_2$ selects the entropy solution for IC1 and ES-$\eta_3$ for IC2, while ES-$\eta_1$ does not select the correct solution for either data set. Enforcing both $\eta_2$ and $\eta_3$ gives the correct solution in both cases.
\end{example}

\begin{figure}[h!]\centering
\begin{subfigure}[t]{.4\linewidth}\centering
\includegraphics[width=\linewidth]{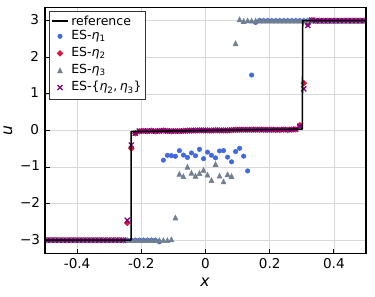}
\caption{IC1: $(u_L,u_R)=(-3,3)$}
\label{fig:num1d-buckley-ic1}
\end{subfigure}\qquad\qquad
\begin{subfigure}[t]{.4\linewidth}\centering
\includegraphics[width=\linewidth]{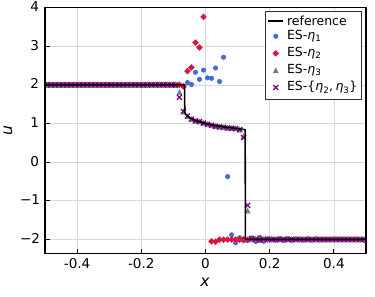}
\caption{IC2: $(u_L,u_R)=(2,-2)$}
\label{fig:num1d-buckley-ic2}
\end{subfigure}
\caption{Example~\ref{ex:buckley-leverett-1d}, $p=4$, $N=80$, $T=1$. In Figure~\ref{fig:num1d-buckley-ic2}, one cell of ES-$\eta_2$, with average $17.64$ at $x=0.0063$, lies above the plotted range.}
\label{fig:num1d-buckley}
\end{figure}

\subsection{One-dimensional Euler equations}
For the one-dimensional Euler equations,
$\bu=(\rho,\rho v,E)^T$ and
$\bfv(\bu)=(\rho v,\rho v^2+P,v(E+P))^T$. We use $\gamma=1.4$ and
$P=(\gamma-1)(E-\rho v^2/2)$, with density $\rho$, velocity $v$, and
total energy $E$. The entropy pair is
$(\eta,\Q)=(-\rho s/(\gamma-1),v\eta)$ with
$s=\log(P\rho^{-\gamma})$. 
Figure~\ref{fig:num1d-euler} collects the density profiles and the corresponding
limiter-activation histories for the next two tests.

\begin{figure}[h!]\centering
\begin{subfigure}[t]{.4\linewidth}\centering
\includegraphics[width=\linewidth]{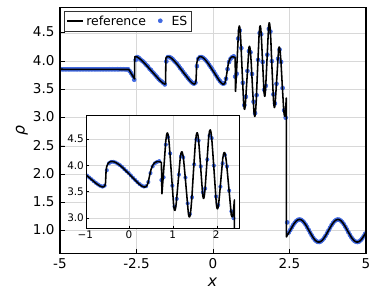}
\caption{Example~\ref{ex:shu-osher}: density}
\label{fig:num1d-euler-shu-density}
\end{subfigure}\qquad\qquad
\begin{subfigure}[t]{.4\linewidth}\centering
\includegraphics[width=\linewidth]{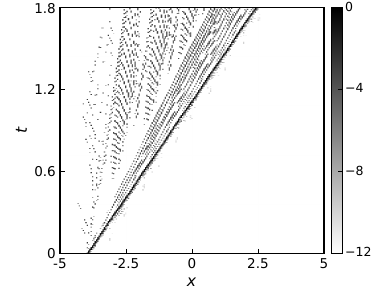}
\caption{Example~\ref{ex:shu-osher}: $\log_{10}(1-\theta_K)$}
\label{fig:num1d-euler-shu-activity}
\end{subfigure}
\qquad
\vspace{2pt}
\begin{subfigure}[t]{.4\linewidth}\centering
\includegraphics[width=\linewidth]{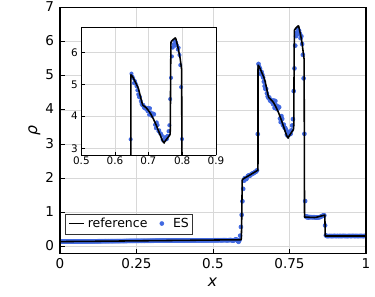}
\caption{Example~\ref{ex:blast-waves}: density}
\label{fig:num1d-euler-blast-density}
\end{subfigure}\qquad\qquad
\begin{subfigure}[t]{.4\linewidth}\centering
\includegraphics[width=\linewidth]{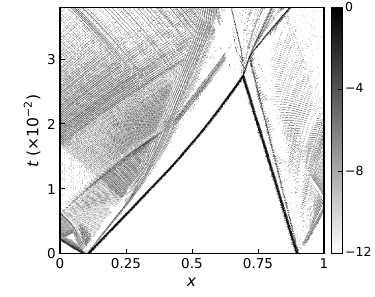}
\caption{Example~\ref{ex:blast-waves}: $\log_{10}(1-\theta_K)$}
\label{fig:num1d-euler-blast-activity}
\end{subfigure}
\caption{Examples~\ref{ex:shu-osher} and~\ref{ex:blast-waves}, $p=4$. Figures~\ref{fig:num1d-euler-shu-density} and~\ref{fig:num1d-euler-shu-activity} show the Shu--Osher problem with $N=200$ and $T=1.8$; Figures~\ref{fig:num1d-euler-blast-density} and~\ref{fig:num1d-euler-blast-activity} show the interacting blast waves with $N=400$ and $T=0.038$. The activity panels show the limiter-activation histories of the entropy correction.}
\label{fig:num1d-euler}
\end{figure}

\begin{example}[Shu--Osher problem]\label{ex:shu-osher}
On $[-5,5]$,
$(\rho,v,P)=(3.857143,\allowbreak 2.629369,\allowbreak 10.3333)$ for $x<-4$ and
$(1+0.2\sin(5x),0,1)$ otherwise \cite{ShuOsher1989}; $N=200$, $T=1.8$,
and no positivity limiter is used. Figure~\ref{fig:num1d-euler-shu-density} shows that the high-frequency waves
behind the shock are resolved. The limiter-activation history in
Figure~\ref{fig:num1d-euler-shu-activity} shows that the entropy correction is
concentrated along the shock path and in the post-shock waves.
\end{example}

\begin{example}[Interacting blast waves]\label{ex:blast-waves}
On $[0,1]$ with reflecting boundaries, $\rho=1$, $v=0$, and
$P=1000$, $0.01$, and $100$ for $x<0.1$, $0.1\leq x\leq0.9$, and
$x>0.9$ \cite{WoodwardColella1984}.
We use $N=400$ cells and compute to $T=0.038$ with the positivity limiter. The narrow density peaks agree well with the reference
in Figure~\ref{fig:num1d-euler-blast-density}. The activation history in
Figure~\ref{fig:num1d-euler-blast-activity} follows the interacting wave pattern.
\end{example}

\subsection{Two-dimensional problems}

\begin{example}[Isentropic vortex]\label{ex:isentropic-vortex} For the 2D Euler equations, 
$\bu=(\rho,\rho v_1,\rho v_2,E)^T$ and
$P=(\gamma-1)\bigl(E-\rho(v_1^2+v_2^2)/2\bigr)$.
We solve the Euler equations with $\gamma=1.4$ on $[0,10]\times[-5,5]$
with periodic boundaries.
A vortex of strength $\beta=5$ centered at $(5,0)$ is added to the mean
flow $(\rho,v_1,v_2,P)=(1,1,0,1)$:
\[
\rho=\Big(1-\frac{(\gamma-1)\beta^2}{16\gamma\pi^2}e^{2(1-r^2)}\Big)^{\frac1{\gamma-1}},\quad
(v_1,v_2)=(1,0)+\frac{\beta}{2\pi}e^{1-r^2}(-y,\,x-5),\quad P=\rho^\gamma,
\]
where $r^2=(x-5)^2+y^2$. We compute to $T=10$, when the solution coincides with the initial data. We use
$\eta=-\rho s/(\gamma-1)$ and
$\boldsymbol{\Q}=(v_1\eta,v_2\eta)$. Table~\ref{tab:num2d-vortex}
shows the designed convergence behavior.
\end{example}

\begin{table}[!htbp]
\centering
\caption{Example~\ref{ex:isentropic-vortex}, $T=10$:
$L^2$ errors of density on
$N\times N$ meshes.}
\label{tab:num2d-vortex}
\normalsize
\begin{tabular}{c r cc cc}
\toprule
& & \multicolumn{2}{c}{Base} & \multicolumn{2}{c}{ES}\\
\cmidrule(lr){3-4}\cmidrule(lr){5-6}
$p$ & $N$ & Error & Order & Error & Order\\
\midrule
\multirow{4}{*}{$2$}
& 20  & $3.176\times10^{-2}$ & --   & $3.181\times10^{-2}$ & --   \\
& 40  & $8.384\times10^{-3}$ & 1.92 & $8.560\times10^{-3}$ & 1.89 \\
& 80  & $1.857\times10^{-3}$ & 2.17 & $1.841\times10^{-3}$ & 2.22 \\
& 160 & $2.770\times10^{-4}$ & 2.74 & $2.751\times10^{-4}$ & 2.74 \\
\midrule
\multirow{4}{*}{$3$}
& 20  & $4.576\times10^{-3}$ & --   & $3.981\times10^{-3}$ & --   \\
& 40  & $2.673\times10^{-4}$ & 4.10 & $2.614\times10^{-4}$ & 3.93 \\
& 80  & $2.813\times10^{-5}$ & 3.25 & $2.802\times10^{-5}$ & 3.22 \\
& 160 & $3.876\times10^{-6}$ & 2.86 & $3.875\times10^{-6}$ & 2.85 \\
\midrule
\multirow{4}{*}{$4$}
& 20  & $5.609\times10^{-4}$ & --   & $5.299\times10^{-4}$ & --   \\
& 40  & $1.934\times10^{-5}$ & 4.86 & $1.907\times10^{-5}$ & 4.80 \\
& 80  & $7.706\times10^{-7}$ & 4.65 & $7.705\times10^{-7}$ & 4.63 \\
& 160 & $3.777\times10^{-8}$ & 4.35 & $3.777\times10^{-8}$ & 4.35 \\
\bottomrule
\end{tabular}
\end{table}

\begin{example}[2D Buckley--Leverett equation]\label{ex:buckley-leverett-2d}
We solve
$u_t+f(u)_x+g(u)_y=0$ with
$f(u)=u^2/(u^2+(1-u)^2)$ and $g(u)=f(u)(1-5(1-u)^2)$ on
$[-1.5,1.5]^2$ with periodic boundaries. The initial state is $u=1$ for
$x^2+y^2<0.5$ and $u=0$ otherwise.
We compute to $T=0.5$ on a $200\times200$ mesh, using the
bound-preserving limiter to enforce $u\in[0,1]$. We enforce either
$\eta_1(u)=u^2/2$ or $\eta_2(u)=\Phi(u-1/2)$, with $\Phi$ from Example~\ref{ex:buckley-leverett-1d}. As shown in Figure~\ref{fig:num2d-bl-eta1}, ES-$\eta_1$ retains some minor nonphysical structures, whereas Figure~\ref{fig:num2d-bl-eta2} shows a more satisfactory result for ES-$\eta_2$.
\end{example}

\begin{figure}[h!]\centering
\begin{subfigure}[t]{.45\linewidth}\centering
\includegraphics[width=\linewidth]{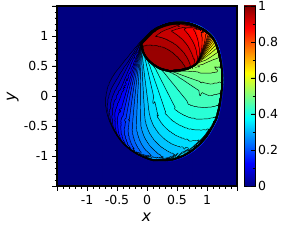}\caption{ES-$\eta_1$}\label{fig:num2d-bl-eta1}
\end{subfigure}\ 
\begin{subfigure}[t]{.45\linewidth}\centering
\includegraphics[width=\linewidth]{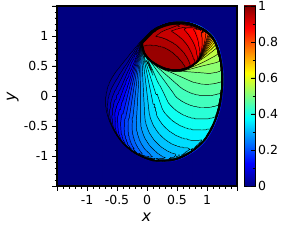}\caption{ES-$\eta_2$}\label{fig:num2d-bl-eta2}
\end{subfigure}
\caption{Example~\ref{ex:buckley-leverett-2d}, $p=4$, $200\times200$, $T=0.5$: cell averages with
$40$ equally spaced contour lines in $(0,1)$. }
\label{fig:num2d-bl}
\end{figure}

\begin{example}[Double Mach reflection]\label{ex:double-mach}
We use the standard Mach-10 double Mach reflection problem
\cite{WoodwardColella1984} on $[0,4]\times[0,1]$. The
incident shock initially satisfies $x=1/6+y/\sqrt3$, with post-shock
state $(\rho,v_1,v_2,P)=(8,8.25\cos30^\circ,-8.25\sin30^\circ,116.5)$
and pre-shock state $(1.4,0,0,1)$. The left boundary is fixed at the
post-shock state, the right boundary at the pre-shock state, and the top
boundary follows the exact incident-shock motion. On $y=0$ we impose the
post-shock state for $x<1/6$ and a reflecting wall for $x\geq1/6$. We use
the entropy pair from Example~\ref{ex:isentropic-vortex} and the positivity
limiter, and compute to $T=0.2$ on meshes of $480\times120$ and
$960\times240$ cells. Figures~\ref{fig:num2d-dm-480-full} and~\ref{fig:num2d-dm-960-full} show the global density profiles, while Figures~\ref{fig:num2d-dm-480-zoom} and~\ref{fig:num2d-dm-960-zoom} show the corresponding Mach-stem regions. The finer mesh more clearly resolves the roll-up of the slip line near the Mach stem, while the main shock geometry is already captured on the coarser mesh.
Oscillations remain observable because no additional
troubled-cell or shock-capturing limiter is used.
\end{example}

\begin{figure}[h!]\centering
\begin{subfigure}[t]{.67\linewidth}\centering
\includegraphics[width=\linewidth,trim=8 0 8 0,clip]{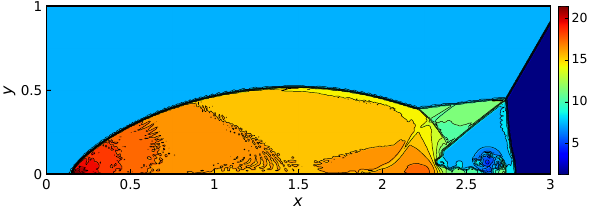}
\caption{$480\times120$}
\label{fig:num2d-dm-480-full}
\end{subfigure}\hspace{2pt}%
\begin{subfigure}[t]{.315\linewidth}\centering
\includegraphics[width=\linewidth,trim=8 0 8 0,clip]{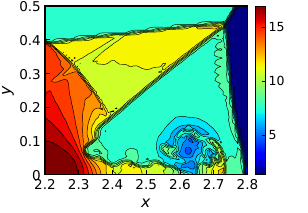}
\caption{$480\times120$, zoom}
\label{fig:num2d-dm-480-zoom}
\end{subfigure}

\vspace{4pt}

\begin{subfigure}[t]{.67\linewidth}\centering
\includegraphics[width=\linewidth,trim=8 0 8 0,clip]{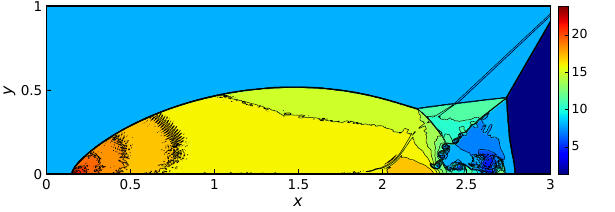}
\caption{$960\times240$}
\label{fig:num2d-dm-960-full}
\end{subfigure}\hspace{2pt}%
\begin{subfigure}[t]{.315\linewidth}\centering
\includegraphics[width=\linewidth,trim=8 0 8 0,clip]{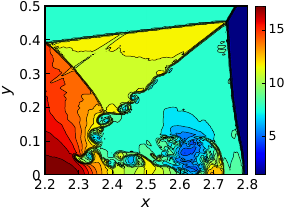}
\caption{$960\times240$, zoom}
\label{fig:num2d-dm-960-zoom}
\end{subfigure}

\caption{Example~\ref{ex:double-mach}, $p=4$, $T=0.2$: density at both resolutions. Figures~\ref{fig:num2d-dm-480-full} and~\ref{fig:num2d-dm-960-full} show the solutions on $[0,3]\times[0,1]$, while Figures~\ref{fig:num2d-dm-480-zoom} and~\ref{fig:num2d-dm-960-zoom} show the corresponding Mach-stem window $[2.2,2.8]\times[0,0.5]$. Each panel uses $20$ equally spaced contour lines between the smallest and largest cell averages in the corresponding window.}
\label{fig:num2d-dm}
\end{figure}

\section{Conclusions}
\label{sec:conclusion}

In this paper, we develop fully discrete multi-entropy-stable
cRKDG methods with high-order explicit RK time discretizations.
We establish a weak cell entropy inequality after one full RK
time step by controlling the cell-local stage perturbations
through the initial in-cell variation and estimating the
numerical-flux contributions. Unlike standard SSP-based arguments, our proof does not rely on
a convex decomposition into forward-Euler steps. This allows us
to apply the weak-to-strong entropy framework to cRKDG schemes
without an SSP requirement. Based on
this weak entropy property, a single cell-average-preserving scaling is
applied at the end of each time step to enforce a prescribed family of
strong cell entropy inequalities simultaneously, while preserving the
conservative cell averages and compact stencil of the cRKDG method. We also prove a Lax--Wendroff-type convergence theorem, showing that a
strongly convergent sequence of numerical solutions converges to a weak
solution satisfying the prescribed entropy inequalities. Numerical
tests confirm the designed high-order accuracy, demonstrate the
benefit of multiple entropy conditions for nonconvex scalar problems,
and show robust performance for  shock problems.

\appendix
\section{Proof of the Lax--Wendroff convergence theorem}\label{app:lax-wendroff}
\begin{proof}
The common CFL condition implies $\tau_h\to0$. All constants below
are independent of $h$ and $\tau_h$; a subscript $q$ indicates possible
dependence on the fixed entropy pair.
\paragraph{Step 1: convergence of the local stages}
Assumption~\ref{ass:state-range} and convexity of $\Bcal$ show that the
post-limiter solutions take values in $\Bcal$; the same is true of
$\bu$ and $\bu_0$ almost everywhere by strong convergence. Set
$M:=\max_{\bw\in\Bcal}|\bw|$. Since $|\bu_h-\bu|\leq2M$,
\begin{equation}\label{eq:LW-L2-convergence}
    \norm{\bu_h-\bu}_{L^2(\Omega_T)}^2
       \leq2M\norm{\bu_h-\bu}_{L^1(\Omega_T)}\longrightarrow0.
\end{equation}
The same argument gives $\bu_h^0\to\bu_0$ in $L^2(\Omega)$.

Cell averaging is an $L^2$ contraction, and for the fixed function
$\bu$ its cell averages satisfy $\overline{\bu}\to\bu$ in
$L^2(\Omega_T)$. Hence
\begin{equation}\label{eq:LW-vanishing-variation}
    \norm{\bu_h-\overline{\bu}_h}_{L^2(\Omega_T)}
    \leq\norm{\bu_h-\bu}_{L^2(\Omega_T)} + \norm{\overline{\bu}_h-\overline{\bu}}_{L^2(\Omega_T)}
       +\norm{\bu-\overline{\bu}}_{L^2(\Omega_T)}\longrightarrow0.
\end{equation}
Define $\bu_h^{(\ell)}(\bx,t):=\bu_h^{n,(\ell)}(\bx)$ on
$[t^n,t^{n+1})$. Applying \eqref{eq:all-stage-L2} cellwise and summing
in space and time gives, for all $\ell=1,\ldots,s$,
\begin{equation}\label{eq:LW-stage-convergence}
\begin{aligned}
    \norm{\bu_h^{(\ell)}-\bu}_{L^2(\Omega_T)}
    &\leq\norm{\bu_h^{(\ell)}-\overline{\bu}_h}_{L^2(\Omega_T)}
       +\norm{\overline{\bu}_h-\overline{\bu}}_{L^2(\Omega_T)}+\norm{\overline{\bu}-{\bu}}_{L^2(\Omega_T)}\longrightarrow0.
\end{aligned}
\end{equation}

\paragraph{Step 2: a weighted face estimate}
To approximate $\bu$ in $L^2(\Omega_T)$, consider a smooth
$\Bcal$-valued function $\bv$ defined on a neighborhood of
$\overline{\Omega_T}$. Let $\overline{\bv}_K(t)$ be its spatial
cell average over $K$. Adding and subtracting $\overline{\bv}_K$
and applying the inverse trace estimate \eqref{eq:inverse-trace},
we obtain
\begin{equation}
\begin{aligned}
 h_K\norm{\bu_h^{(\ell),-}-\bv}_{\partial K}^2
 &\le C\norm{\bu_h^{(\ell)}-\overline{\bv}_K}_K^2
   +C h_K\norm{\bv-\overline{\bv}_K}_{\partial K}^2 \\
 &\le C\norm{\bu_h^{(\ell)}-\bv}_K^2
   +C h_K^2|K|\norm{\nabla\bv}_{L^\infty}^2,
\end{aligned}    
\end{equation}
where the second line uses
$|\bv-\overline{\bv}_K|\le h_K\norm{\nabla\bv}_{L^\infty}$ and
\eqref{eq:cell-surface-ratio}.
Thus \eqref{eq:LW-stage-convergence} yields
\begin{equation}\label{eq:LW-weighted-face-limsup}
    \limsup_{h\to0}\int_0^T\sum_K h_K
       \norm{\bu_h^{(\ell),-}-\bv}_{\partial K}^2\dd t
       \leq C\norm{\bu-\bv}_{L^2(\Omega_T)}^2.
\end{equation}
Here $C$ is independent of $\bv$; no trace of the limiting function
$\bu$ is used.

\paragraph{Step 3: consistency of the RK numerical fluxes}
Fix $\varphi\in C_c^\infty(\Omega\times[0,T))$ and set
\begin{equation}\label{eq:LW-test-averages}
    \overline\varphi_K^{\,n}:=
       \fint_K\varphi(\bx,t^n)\dd\bx,
    \qquad
    \varphi_{h,\tau}|_{K\times[t^n,t^{n+1})}
       :=\overline\varphi_K^{\,n+1}.
\end{equation}
For $h$ small enough, these averages vanish on cells meeting
$\partial\Omega$, so physical boundary faces do not contribute below.
For fixed $q\in\mathcal I$, define
\begin{equation}\label{eq:LW-entropy-flux-sum}
    S_{h,q}(\varphi):=
       \sum_{n=0}^{N_h-1}\tau_h\sum_K
          \overline\varphi_K^{\,n+1}
          \int_{\partial K}\wh{\Q}_{q,\bn_K}^{n,\mathrm{RK}}\dd s.
\end{equation}
In the remaining part of this step, we prove
\begin{equation}\label{eq:LW-entropy-flux-limit}
    S_{h,q}(\varphi)\longrightarrow
       -\int_0^T\int_\Omega
          \boldsymbol{\mathcal Q}_q(\bu)\cdot\nabla\varphi\dd\bx\dd t.
\end{equation}

As $\bu$ may not have sufficient regularity to define traces on cell faces, we introduce a smooth comparison function $\bv$. Set
\begin{equation}\label{eq:LW-smooth-flux-sum}
\begin{aligned}
S_{h,q}[\bv](\varphi)
&:=\int_0^T\int_\Omega
\varphi_{h,\tau}\,
\nabla\cdot\boldsymbol{\mathcal Q}_q(\bv)
\dd\bx\dd t \\
&=
\sum_{n=0}^{N_h-1}\int_{t^n}^{t^{n+1}}
\sum_K \overline{\varphi}_K^{\,n+1}
\int_{\partial K}\Q_{q,\bn_K}(\bv)\dd s\dd t.
\end{aligned}
\end{equation}
Here the second identity comes from the divergence theorem. On an interior face $e=K|L$, write
$\wh{\Q}_{q,e}^{n,(\ell)}:=
\wh{\Q}_{q,\bn_K}(\bu_{h,K}^{n,(\ell),-},\bu_{h,L}^{n,(\ell),-})$.
Subtracting \eqref{eq:LW-smooth-flux-sum} from
\eqref{eq:LW-entropy-flux-sum} and regrouping the interface integrals
over the shared faces $e=K|L$ yields
\begin{equation}\label{eq:LW-flux-difference-RK}
\begin{aligned}
&S_{h,q}(\varphi)-S_{h,q}[\bv](\varphi)\\
&=\sum_{n=0}^{N_h-1}\int_{t^n}^{t^{n+1}}
   \sum_{e=K|L}\int_e
   (\overline\varphi_K^{\,n+1}-\overline\varphi_L^{\,n+1})\sum_{\ell=1}^s b_\ell
   [\wh{\Q}_{q,e}^{n,(\ell)}-\Q_{q,\bn_K}(\bv)]\dd s\dd t .
\end{aligned}
\end{equation}
Here we have used $\bn_K = -\bn_L$, conservation of $\wh{\Q}$, and $\sum_\ell b_\ell=1$.
The mean-value theorem and quasi-uniformity imply
\begin{equation}\label{eq:LW-test-jump}
\begin{aligned}
|\overline\varphi_K^{\,n+1}-\overline\varphi_L^{\,n+1}|
&\leq\norm{\nabla\varphi}_{L^\infty}
   \fint_K\fint_L|\bx-\by|\dd\by\dd\bx
&\leq(h_K+h_L)\norm{\nabla\varphi}_{L^\infty}
 \leq C_\varphi h_K.
\end{aligned}
\end{equation}
By consistency and Lipschitz continuity of the numerical entropy flux,
\begin{equation}\label{eq:LW-stage-flux-Lipschitz}
    |\wh{\Q}_{q,e}^{n,(\ell)}-\Q_{q,\bn_K}(\bv)|
    \leq C_q\bigl(
       |\bu_{h,K}^{n,(\ell),-}-\bv|
       +|\bu_{h,L}^{n,(\ell),-}-\bv|\bigr).
\end{equation}
Combining the last three estimates and using $b_\ell\geq0$ gives
\begin{equation}\label{eq:LW-face-flux-comparison}
\begin{aligned}
|S_{h,q}(\varphi)-S_{h,q}[\bv](\varphi)|
&\leq C_{q,\varphi}\sum_{\ell=1}^s b_\ell
 \int_0^T\sum_{e=K|L}h_K\int_e
 \bigl(|\bu_{h,K}^{(\ell),-}-\bv|
      +|\bu_{h,L}^{(\ell),-}-\bv|\bigr)\dd s\dd t\\
&\leq C_{q,\varphi}\sum_{\ell=1}^s b_\ell
 \int_0^T\sum_K h_K\int_{\partial K}
 |\bu_h^{(\ell),-}-\bv|\dd s\dd t.
\end{aligned}
\end{equation}
For the second inequality, the $L$-trace is the interior trace of the
neighboring cell and quasi-uniformity gives $h_K\leq C h_L$.
Cauchy--Schwarz and \eqref{eq:cell-surface-ratio} now yield
\begin{equation}\label{eq:LW-flux-comparison-CS}
\begin{aligned}
\int_0^T\!\sum_K h_K\!\int_{\partial K}
|\bu_h^{(\ell),-}-\bv|\dd s\dd t
&\leq(C_kT|\Omega|)^{1/2}
\left(\int_0^T\!\sum_K h_K
\norm{\bu_h^{(\ell),-}-\bv}_{\partial K}^2\dd t\right)^{1/2}.
\end{aligned}
\end{equation}
Therefore \eqref{eq:LW-weighted-face-limsup} implies
\begin{equation}\label{eq:LW-flux-comparison-limsup}
    \limsup_{h\to0}|S_{h,q}(\varphi)-S_{h,q}[\bv](\varphi)|
    \leq C_{q,\varphi}\norm{\bu-\bv}_{L^2(\Omega_T)},
\end{equation}
where $C_{q,\varphi}$ is independent of $\bv$.

For fixed smooth $\bv$, the estimate
$\norm{\varphi_{h,\tau}-\varphi}_{L^\infty}
\leq C_\varphi(h+\tau_h)$ gives
\begin{equation}\label{eq:LW-smooth-flux-limit}
\begin{aligned}
S_{h,q}[\bv](\varphi)
&=\int_0^T\int_\Omega\varphi_{h,\tau}
   \nabla\cdot\boldsymbol{\mathcal Q}_q(\bv)\dd\bx\dd t\longrightarrow-\int_0^T\int_\Omega
   \boldsymbol{\mathcal Q}_q(\bv)\cdot\nabla\varphi\dd\bx\dd t,
\end{aligned}
\end{equation}
because $\varphi$ vanishes near $\partial\Omega$. Hence
\begin{equation}\label{eq:LW-entropy-flux-limsup-v}
\begin{aligned}
\limsup_{h\to0}\left|
S_{h,q}(\varphi)+\int_0^T\int_\Omega
\boldsymbol{\mathcal Q}_q(\bv)\cdot\nabla\varphi\dd\bx\dd t\right|
\leq C_{q,\varphi}\norm{\bu-\bv}_{L^2(\Omega_T)}.
\end{aligned}
\end{equation}
Lipschitz continuity of $\boldsymbol{\mathcal Q}_q$ on $\Bcal$ also gives
\begin{equation}\label{eq:LW-smooth-entropy-flux-approximation}
\begin{aligned}
\left|\int_0^T\int_\Omega
[\boldsymbol{\mathcal Q}_q(\bv)-\boldsymbol{\mathcal Q}_q(\bu)]
\cdot\nabla\varphi\dd\bx\dd t\right|
\leq C_{q,\varphi}\norm{\bu-\bv}_{L^2(\Omega_T)}.
\end{aligned}
\end{equation}
Approximating $\bu$ by smooth $\Bcal$-valued functions $\bv$ in
$L^2(\Omega_T)$ proves \eqref{eq:LW-entropy-flux-limit}. The same
argument applied componentwise to the physical numerical flux gives
\begin{equation}\label{eq:LW-physical-flux-limit}
\begin{aligned}
&\sum_{n=0}^{N_h-1}\tau_h\sum_K\overline\varphi_K^{\,n+1}
 \int_{\partial K}\wh{\bfv}_{\bn_K}^{n,\mathrm{RK}}\dd s \longrightarrow
 -\int_0^T\int_\Omega\sum_{r=1}^d
 \bfv_r(\bu)\,\partial_{x_r}\varphi\dd\bx\dd t.
\end{aligned}
\end{equation}

\paragraph{Step 4: the weak conservation law}
Set
$\delta_t\overline\varphi_K^{\,n}:=
(\overline\varphi_K^{\,n+1}-\overline\varphi_K^{\,n})/\tau_h$.
Multiplying \eqref{eq:post-limiter-average} by
$|K|\overline\varphi_K^{\,n+1}$ and summing by parts in time gives
\begin{equation}\label{eq:LW-discrete-conservation}
\begin{aligned}
&\sum_{n=0}^{N_h-1}\tau_h\sum_K
 |K|\overline{\bu}_{h,K}^{\,n}\delta_t\overline\varphi_K^{\,n}
 +\sum_K|K|\overline{\bu}_{h,K}^{\,0}\overline\varphi_K^{\,0}-
 \sum_{n=0}^{N_h-1}\tau_h\sum_K\overline\varphi_K^{\,n+1}
 \int_{\partial K}\wh{\bfv}_{\bn_K}^{n,\mathrm{RK}}\dd s=\mathbf0.
\end{aligned}
\end{equation}
The first term equals
$\int_0^T\int_\Omega
       \bu_h\,\partial_t\overline\varphi\dd\bx\dd t$, 
where $\overline\varphi$ denotes the spatial cell average of
$\varphi(\cdot,t)$ at each time $t$. Since
$\partial_t\overline\varphi=\overline{\partial_t\varphi}$ converges
uniformly to $\partial_t\varphi$, this term has the required limit.
The initial term converges similarly, and
\eqref{eq:LW-physical-flux-limit} then proves
\eqref{eq:LW-weak-solution}.

\paragraph{Step 5: the prescribed entropy inequalities}
Fix $q\in\mathcal I$ and $\varphi\geq0$, and set
$E_{K,q}^n:=\int_K\eta_q(\bu_h^n)\dd\bx$. Multiplying
\eqref{eq:strong-cell-entropy} by
$|K|\overline\varphi_K^{\,n+1}$ and summing by parts gives
\begin{equation}\label{eq:LW-discrete-entropy}
    \sum_{n=0}^{N_h-1}\tau_h\sum_K
       E_{K,q}^n\delta_t\overline\varphi_K^{\,n}
    +\sum_K E_{K,q}^0\overline\varphi_K^{\,0}
    -S_{h,q}(\varphi)\geq0.
\end{equation}
The first two terms equal
$\int_0^T\int_\Omega
       \eta_q(\bu_h)\,\partial_t\overline\varphi\dd\bx\dd t
       +\int_\Omega\eta_q(\bu_h^0)\overline\varphi(\bx,0)\dd\bx$. 
For this fixed $q$, Lipschitz continuity of $\eta_q$ on $\Bcal$ and
\eqref{eq:LW-strong-convergence} give convergence of these two terms,
while the flux term converges by \eqref{eq:LW-entropy-flux-limit}.
Passing to the limit proves \eqref{eq:LW-entropy-solution}. The same
argument applies to every fixed $q\in\mathcal I$ along the original
sequence; no entropy-dependent subsequence or uniform-in-$q$
Lipschitz bound is required.
\end{proof}

\textbf{Acknowledgements or Declarations}
This work was supported by NSF grant DMS-2208391. The author also
acknowledges OpenAI for providing access to ChatGPT and Codex through
the ChatGPT for Academic Researchers program, and Anthropic for
providing access to Claude through the Claude Team plan for research
labs. Generative AI tools (OpenAI's ChatGPT and Codex; Anthropic's Claude)
were used for mathematical discussions, manuscript drafting and revision,
LaTeX preparation, and the development, execution, and checking of
computer code for the numerical experiments, including the generation
of figures and tables. The author reviewed all content and reported results and takes
full responsibility for the manuscript.


\end{document}